\documentclass{amsart}

\usepackage{bbm, amssymb, amsthm, verbatim}

\allowdisplaybreaks[4]

\newcommand{\pD}[2]{\frac{\partial #1}{\partial #2}}

\newcommand{\rD}[2]{\frac{d #1}{d #2}}

\newcommand{\vn}[1]{\lVert#1\rVert}

\newcommand{\nIP}[2]{\left< #1 , #2 \right>}
\newcommand{\sIP}[2]{\left(\left. #1 \right| #2 \right)}
\newcommand{\sfrac}[2]{\text{\fontsize{5}{5}\selectfont$\frac{#1}{#2}$}}

\newcommand{\R}{\ensuremath{\mathbb{R}}}

\newcommand{\hvarepsilon}{\hat\varepsilon_0}
\newcommand{\hsigma}{\hat\sigma}
\newcommand{\hosigma}{\hat\sigma}

\newcommand{\RMn}{\ensuremath{\overline{R}}}

\newcommand{\SW}{\ensuremath{\mathcal{W{}}}}
\newcommand{\cK}{\ensuremath{\mathcal{\overline{K}}}}
\newcommand{\SK}{\ensuremath{\mathcal{K{}}}}

\newcommand{\veZero}{\ensuremath{\varepsilon_1}}
\newcommand{\veOne}{\ensuremath{\varepsilon_2}}
\newcommand{\veTwo}{\ensuremath{\varepsilon_3}}
\newcommand{\veThree}{\ensuremath{\varepsilon_0}}
\newcommand{\BWv}{\ensuremath{\mathbf{W}}}
\newcommand{\BF}{\ensuremath{\mathbf{F}}}

\newcommand{\Rc}{\text{\rm Ric}}
\newcommand{\cg}{\ensuremath{c_{\gamma}}}
\newcommand{\csupp}{\ensuremath{c_{\text{supp}}}}
\newcommand{\gconc}{\ensuremath{\int_\Sigma \gamma^4\,d\mu}}
\newcommand{\Rcn}{\overline{\text{\rm Ric}}}
\newcommand{\nablan}{\overline{\nabla}}
\newcommand{\Sc}{{\text{\rm Sc}}}
\newcommand{\Scn}{\overline{\text{\rm Sc}}}

\newtheorem{thm}{Theorem}
\newtheorem{cor}[thm]{Corollary}
\newtheorem{prop}[thm]{Proposition}
\newtheorem{lem}[thm]{Lemma}

\theoremstyle{remark}
\newtheorem{rmk}{Remark}

\begin{document}

\title[Willmore flow of surfaces in Riemannian spaces I]{Willmore flow of surfaces in
Riemannian spaces I: Concentration-compactness}
\author{Jan Metzger
   \and Glen Wheeler
   \and Valentina-Mira Wheeler$^*$
   }
\thanks{* Corresponding author, \texttt{vwheeler@uow.edu.au}}
\address{Jan Metzger \& Valentina-Mira Wheeler (previous)\\Intit\"ut f\"ur Mathematik\\
         Universit\"at Potsdam\\
         Am Neuen Palais 10\\
         14469, Potsdam, Germany}
\address{Glen Wheeler (previous)\\Fakult\"at f\"ur Mathematik\\
         Intit\"ut f\"ur Analysis und Numerik\\
         Otto-von-Guericke Universit\"at Magdeburg\\
         Universit\"atsplatz 2\\
         39106, Magdeburg, Germany}
\address{Glen Wheeler \& Valentina-Mira Wheeler\\Institute for Mathematics and its Applications\\
         Faculty of Informatics and Engineering\\
         University of Wollongong\\
         Northfields Avenue, 2522\\
         Wollongong, NSW, Australia}

\begin{abstract}
In this paper we study the local regularity of
closed
surfaces immersed in a Riemannian ambient space $(N^3,\nIP{\cdot}{\cdot}$) flowing by Willmore flow.
We establish a pair of concentration-compactness alternatives for the flow, giving a lower bound on the maximal time of existence of the flow proportional to
the concentration of the curvature and area at initial time.
The estimate from the first theorem is purely in terms of the concentration of curvature at initial time but applies only to ambient spaces with non-positive
sectional curvature.
The second requires additional information on the concentration of area at initial time but applies in more general background spaces.
Applications of these results shall appear in forthcoming work.
\end{abstract}

\maketitle

\section{Introduction}


Suppose $f:\Sigma\rightarrow N$ is a smooth immersion of the surface $\Sigma$ into the smooth
three-dimensional complete Riemannian manifold $(N,\nIP{\cdot}{\cdot})$.  Let us equip $\Sigma$ with the metric
induced by $f$, so that $(\Sigma, f^*\nIP{\cdot}{\cdot})$ is a Riemannian manifold.
We assume that $(\Sigma, f^*\nIP{\cdot}{\cdot})$ is closed and complete.
Consider the functional
\[
\SW(f) = \frac14\int_\Sigma H^2\,d\mu\,,
\]
where $H$ is the mean curvature and $d\mu$ the measure induced via $f$.
Let us use $A$ to denote the second fundamental form of the immersion $f$, and $A^o$ to denote its tracefree part.
Surfaces which are critical for $\SW$ satisfy the Euler-Lagrange equation
\begin{equation}
\label{EQeulerlagrangevanilla}
\BWv(f) = \Delta H + H|A^o|^2 + H\Rcn(\nu,\nu) = 0\,,
\end{equation}
where $\Rcn$ denotes the Ricci curvature of $(N,\nIP{\cdot}{\cdot})$ and $\nu$ is the exterior unit normal vectorfield along $f$.
Note that \eqref{EQeulerlagrangevanilla} is invariant under change of orientation and reparametrisation.

In this work we study the steepest descent $L^2$-gradient flow of $\SW$, termed the Willmore flow.
These are one-parameter families of immersions $f:\Sigma\times[0,T)\rightarrow N$ satisfying
\begin{equation}
\label{EQwf}
\pD{}{t}f = -\BWv(f)\nu = -\big(\Delta H + H|A^o|^2 + H\Rcn(\nu,\nu)\big)\nu\,.
\end{equation}
We state the following local existence result.
The proof follows by first writing the solution as a graph over the initial manifold in the direction of the unit normal $\nu$ and then applying the standard
theory of higher-order degenerate quasilinear parabolic equations.
Details can be found in \cite[Chapter 3]{bakerthesis}.
See also \cite[Chapter 5]{eidelman1998pbv}, \cite{solonnikov1965bvp} and \cite{shuanhuli}.
\begin{thm}[Local existence for Willmore flow in a Riemannian space]
\label{TMstevanilla}
Suppose $f_0:\Sigma\rightarrow N$ is a smooth immersed surface.
Assume that the induced Riemannian structure $(\Sigma, f^*\nIP{\cdot}{\cdot})$ is complete and
closed.
Then there exists a maximal $T>0$ and a one-parameter family of smooth immersions
$f:\Sigma\times[0,T)\rightarrow N$ such that
\begin{align*}
\pD{}{t}f &= -\BWv(f)\nu\,,\quad\text{and}
\\
f(\cdot,0) &= f_0(\cdot)\,.
\end{align*}
The family $f$ is called a \emph{Willmore flow with initial data $f_0$}.
\end{thm}
Our main result is the following pair of concentration-compactness alternatives for the flows \eqref{EQwf}, sometimes termed \emph{Lifespan Theorems} \cite{MWW10,W11}.

\begin{thm}
\label{TMlifespanvanilla}
Let $(N,\nIP{\cdot}{\cdot})$ be a smooth simply-connected Riemannian 3-manifold with non-positive sectional curvature.
That is,
\begin{equation}
\label{EQambeintassumptionsforlifespan}
\SK :=
\sup_{p\in N,\,X,Y\in T_pN} \Rcn(X,Y)(p)
\le 0
\,.
\end{equation}
Above we have used $\Rcn$ to denote the Ricci curvature of $(N,\nIP{\cdot}{\cdot})$.
Suppose $f:\Sigma\times[0,T)\rightarrow N$ is a one-parameter family of closed immersed surfaces with smooth initial
data evolving by \eqref{EQwf}.
Then there are constants $\rho>0$, $\veThree>0$, and $c<\infty$ depending only on the metric of $N$ and $\vn{\overline{\nabla_{(k)}\Rc}}_\infty$
for $k\in\{0,1,2,3,4,5\}$ such that if $\rho$ is chosen with
\begin{equation}
\label{EQsmallconcentrationcondition}
\int_{f^{-1}(B_\rho(x))} |A|^2 d\mu\Big|_{t=0} = \varepsilon(x) \le \veThree
\qquad
\text{ for all $x\in N$},
\end{equation}
then the maximal time $T$ of smooth existence satisfies
\begin{equation}
\label{EQmaximaltimeestimatevanilla}
T \ge \frac{1}{c}\rho^4,
\end{equation}
and we have the estimate
\begin{equation}
\label{EQconcentrationestimatevanilla}
\int_{f^{-1}(B_\rho(x))} |A|^2 d\mu \le c\varepsilon(x)
\qquad\qquad\qquad\qquad\hskip-1mm\text{ for all }
t \in \Big[0, \frac{1}{c}\rho^4\Big].
\end{equation}
\end{thm}

\begin{rmk}[Dependence on the metric of $N$]
Covering arguments are used in the proof of Theorem \ref{TMlifespanvanilla} above and Theorem
\ref{TMlifespanvanillapositivecurvature} below. Here the number of balls needed to cover a ball with
a given radius in $N$ depends on the metric locally around the largest ball. In addition to bounds
on the curvature of $N$ and the injectivity radius of $N$, this is the only way in which the
constants depend on the metric of $N$.  Note for Theorem \ref{TMlifespanvanilla} the hypotheses
imply that the injectivity radius of $N$ is unbounded.
\end{rmk}

In spaces with some positive curvature, the required Sobolev inequality fails for relatively `large' submanifolds.  In order to compensate for this,
we supplement the argument used for Theorem \ref{TMlifespanvanilla} with control of the concentration of area along the flow.  Our result in this case
is the following.

\begin{thm}[Lifespan theorem for Willmore flow in a Riemannian 3-manifold]
\label{TMlifespanvanillapositivecurvature}
Let $(N,\nIP{\cdot}{\cdot})$ be a smooth Riemannian 3-manifold with positive injectivity radius $\rho_N$.
Suppose $f:\Sigma\times[0,T)\rightarrow N$ is a one-parameter family of immersed surfaces with smooth initial
data evolving by \eqref{EQwf}.
Then there are constants $\rho>0$, $\hvarepsilon>0$, $\hsigma_0>0$ and $c<\infty$ depending only on the metric of $N$
and
$\vn{\overline{\nabla_{(k)}\Rc}}_\infty$ for $k\in\{0,1,2,3,4,5\}$ such that if $\rho$ is chosen with
\begin{equation*}
\int_{f^{-1}(B_\rho(x))} |A|^2 d\mu\Big|_{t=0} = \varepsilon(x) \le \hvarepsilon
\qquad
\text{ for any $x\in N$, and}
\end{equation*}
\begin{equation}
\label{EQsmallconcentrationconditionpositivecurvature}
\big|\Sigma\big|_{f^{-1}(B_\rho(x))}\big|_{t=0} := \int_{f^{-1}(B_\rho(x))} d\mu\Big|_{t=0} = \sigma(x) \le \hsigma_0
\qquad
\text{ for any $x\in N$},
\end{equation}
then the maximal time $T$ of smooth existence satisfies
\begin{equation}
\label{EQmaximaltimeestimate}
T \ge \frac{1}{c}\rho^4,
\end{equation}
and we have the estimate
\begin{equation}
\label{EQconcentrationestimate}
\int_{f^{-1}(B_\rho(x))} |A|^2 d\mu
+ \big|\Sigma\big|_{f^{-1}(B_\rho(x))}
 \le c(\varepsilon(x) + \sigma(x))
\qquad\text{ for }\qquad
0\le t \le \frac{1}{c}\rho^4.
\end{equation}
\end{thm}

\begin{rmk}
The universal constants $\veThree$, $\hvarepsilon$ and $\sigma$ are, given $N$, computable and not the result of an abstract existence proof.
\end{rmk}

In a more global sense, we present Theorems \ref{TMlifespanvanilla} and \ref{TMlifespanvanillapositivecurvature} with a perspective toward further analysis of
the flow \eqref{EQwf}.
In particular, as the statement depends on the concentration of the curvature of the initial surface, the result is particularly relevant to the analysis of
asymptotic behaviour in the following respect.  When considering a blowup of a singularity formed at some time $T<\infty$ of the constrained Willmore flow, we
wish to have that some amount of the curvature concentrates in space. For example, from Theorem \ref{TMlifespanvanilla}, if $\rho(t)$ denotes the
largest radius such that \eqref{EQsmallconcentrationcondition} holds at time $t$, then $\rho(t) \le \sqrt[4]{c(T-t)}$ and so at least
$\veThree$ of the curvature concentrates in a ball $f^{-1}(B_{\rho(T)}(x))$.  That is,
\[
\lim_{t\rightarrow T} \int_{f^{-1}(B_{\rho(t)}(x))} |A|^2 d\mu \ge \veThree\,,
\]
where $x = x(t)$ is understood to be the centre of a ball where the integral above is maximised.  This will be a fundamental property of blowups considered in
an upcoming paper.

%


\section{Notation and setting.}

In this section we collect various general formulae from the differential geometry of submanifolds which we need for later analysis.  We use notation
similar to that of Hamilton \cite{RH} and Huisken \cite{huisken1984fmc,huisken86riemannian}. We have as our principal object of study a smooth
immersion $f:\Sigma\rightarrow (N,\nIP{\cdot}{\cdot})$ of an orientable surface $\Sigma$ into a Riemannian 3-manifold $(N,\nIP{\cdot}{\cdot})$.  The
definitions below are understood with respect to a local orthonormal frame $\{e_0,e_1,e_2\}$ of $N$, such that restricted to $\Sigma$ we have $e_0 =
\nu$ and $e_i = \pD{f}{x_i}$ for $i=1,2$.  Further, we take this frame to induce normal coordinates on $\Sigma$, so that the Christoffel symbols on
$\Sigma$ and $N$ vanish at a single point.  Throughout the paper we perform all calculations with respect to this local orthonormal frame and these
coordinates.

The immersion induces a Riemannian metric on $\Sigma$ with components
\[
g_{ij} = \nIP{\pD{}{x_i}f}{\pD{}{x_j}f},
\]
so that the pair $(M,g)$ is a Riemannian submanifold of $(N,\nIP{\cdot}{\cdot})$.
We use
\[
\sIP{X}{Y} = g_{ij}X^iY^j
\]
to denote the inner product between $X$ and $Y$ with respect to the metric $g$.
Here and for the rest of the paper we have adopted the convention that repeated indices are summed from 1 to 2.

The Riemannian metric induces an inner product structure on all tensor fields over $\Sigma$, not just vector
fields.
This is realised as the trace over pairs of indices with the metric:
\[
\sIP{T^{i}_{jk}}{S^i_{jk}} = g_{is}g^{jr}g^{ku}T^i_{jk}S^s_{ru},\qquad |T|^2 = \sIP{T}{T}\,.
\]
In the above formula $g^{ij} = (g^-1)_{ij}$ where $g^-1$ is the inverse of $g$.
For tensors $\overline{T}$ defined purely on $N$, we use the convention that $|\overline{T}|^2 =
\nIP{\overline{T}}{\overline{T}}$ unless otherwise specified.
The mean curvature $H$ is given by
\[
H = g^{ij}A_{ij} = A_i^i,
\]
where the components $A_{ij}$ of the second fundamental form $A$ are
\begin{equation}
A_{ij} = \nIP{\nablan_j\pD{}{x_i}f}{\nu}\,.
\label{EQsff}
\end{equation}
Here we use $\nablan$ to refer to the Levi-Civita connection on $N$.
Throughout the paper we use a bar to indicate quantities on $N$.
We note that this is the opposite sign convention of Huisken \cite{huisken1984fmc,huisken86riemannian}, and
the same sign convention as Kuwert \& Sch\"atzle \cite{kuwert2001wfs,kuwert2002gfw}.
The Christoffel symbols of the induced connection are determined by the metric,
\begin{equation}
\label{C3Echristoffelmetric}
\Gamma_{ij}^k = \frac{1}{2}g^{kl}
                \left(\pD{}{x_i}g_{jl} + \pD{}{x_j}g_{il} - \pD{}{x_l}g_{ij}\right),
\end{equation}
with the covariant derivative on $\Sigma$ of a vector $X$ and of a covector $Y$ is
\begin{align*}
\nabla_jX^i &= \pD{}{x_j}X^i + \Gamma^i_{jk}X^k\text{, and}\\
\nabla_jY_i &= \pD{}{x_j}Y_i - \Gamma^k_{ij}Y_k
\end{align*}
respectively.

The second fundamental form is symmetric and satisfies the Codazzi equations:
\[
\nabla_iA_{jk} - \nabla_jA_{ik} = \RMn_{0ijk},
\]
Here $\RMn$ denotes the curvature tensor of $N$
\[
\RMn_{ijkl}{\overline{g}}^{lm}\partial_m = \big(\nablan_{ij} - \nablan_{ji}\big)\partial_k\,.
\]
The curvature tensor $R$ on $\Sigma$ is defined analogously.
The second fundamental relation between components of the Riemann curvature tensor $R$ on $\Sigma$, the
curvature tensor $\RMn$ on $N$, and the second fundamental form $A$ on $\Sigma$, is given by Gauss' equation
\begin{align*}
R_{ijkl} = \RMn_{ijkl} &+ A_{il}A_{jk} - A_{ik}A_{jl},\intertext{with contractions}
g^{jk}R_{ijkl}
   = \Rc(e_i,e_l)
  &= g^{jk}\RMn_{ijkl} + HA_{il} - A_{ij}A_l^j\\
  &= \Rcn(e_i,e_l) - \Rcn(\nu,\nu) + HA_{il} - A_{ij}A_l^j\,,\quad\text{and}\\
g^{il}\Rc(e_i,e_l)
   = \Sc
  &= g^{il}g^{jk}\RMn_{ijkl} + H^2 - |A|^2
\\
  &= \Scn - 2\Rcn(\nu,\nu) + H^2 - |A|^2\,,
\end{align*}
where $\Sc$, $\Scn$ is the scalar curvature of $(\Sigma,g)$ and $(N,\nIP{\cdot}{\cdot})$ respectively.
We will need to interchange covariant derivatives; for vectors $X$ and covectors $Y$ we obtain
\begin{align*}
\nabla_{ij}X^h - \nabla_{ji}X^h &= R_{lijk}g^{hl}X^k
                                 = \RMn_{lijk}g^{hl}X^k - (A_{lj}A_{ik}-A_{lk}A_{ij})g^{hl}X^k\,,\\
\nabla_{ij}Y_k - \nabla_{ji}Y_k &= R_{ijkl}g^{lm}Y_m
                                 = \RMn_{ijkl}g^{lm}Y_m - (A_{lj}A_{ik}-A_{il}A_{jk})g^{lm}Y_m\,,
\end{align*}
where $\nabla_{i_1\ldots i_n} = \nabla_{i_1} \cdots \nabla_{i_n}$.  Further for $T$ a tensor field of type
$(p,q)$ we define $\nabla_{(r)}T$ to be the tensor field of order $(p,q+r)$ with components $\nabla_{i_1\ldots
i_r}T_{j_1\ldots j_p}^{k_1\ldots k_q}$.
We also use for tensors $T$ and $S$ the notation $T*S$ (as in Hamilton \cite{RH}) to denote a linear
combination of new tensors, each formed by contracting pairs of indices from $T$ and $S$ by the metric $g$
with multiplication by a universal constant.  The resultant tensor will have the same type as the other
quantities in the equation it appears.  Keeping these in mind we also denote polynomials in the
iterated covariant derivatives of these terms by
\[
P_j^i(T) = \sum_{k_1+\ldots+k_j = i} c_{k_1\cdots k_j}\nabla_{(k_1)}T*\cdots*\nabla_{(k_j)}T\,,
\]
where $k_1,\ldots,k_j\ge0$ and the constants $c_{k_1\cdots k_j}\in\R$ are absolute.
We use the convention that $P_j^i(T) = 0$ if $i<0$ or $j\le0$.
At times we will need to consider polynomials in multiple tensors or functions, for which we shall use the
notation
\begin{align*}
P_{\alpha,j}^i(S,T) = \sum_{k_1+\ldots+k_j+\beta_1+\ldots+\beta_\alpha = i}
                      c_{k_1\cdots k_j\beta_1\cdots \beta_\alpha}
                      &\Big(\nabla_{(\beta_1)}S*\cdots*\nabla_{(\beta_\alpha)}S\Big)
\\
                     *&\Big(\nabla_{(k_1)}T*\cdots*\nabla_{(k_j)}T\Big)\,,
\end{align*}
where $k_1,\ldots,k_j,\beta_1,\ldots,\beta_\alpha\ge0$ and the constants $c_{k_1\cdots k_j\beta_1\cdots\beta_\alpha}\in\R$ are absolute.
As is common for the $*$-notation, we slightly abuse these constants when certain subterms do not appear in
our $P$-style terms.
For example
\begin{align*}
|\nabla A|^2    &= \sIP{\nabla A}{\nabla A}
	\\	&= 1\cdot\left(\nabla_{(1)}A*\nabla_{(1)}A\right)
                    + 0\cdot\left(\nabla_{(2)}A*A\right)
                    + 0\cdot\left(A*\nabla_{(2)}A\right)
		= P_2^2(A)\,.
\end{align*}
This will occur throughout the paper without further comment.

The Laplace-Beltrami operator on $\Sigma$ acting on a tensor $T$ is given by
\[
\Delta T^i_{jk} = g^{pq}\nabla_{pq}T^i_{jk} = \nabla^p\nabla_pT^i_{jk}.
\]
Using the Codazzi equation with the interchange of covariant derivative formula given above, we
obtain Simons' identity \cite[Lemma 2.1]{huisken86riemannian}:
\begin{align*}
\Delta A_{ij}
 &=  \nabla_{ij}H
   + HA_{ir}A^r_j
   - A_{ij}|A|^2
   + \nabla_q\RMn_{0ij}{}^q
   + \nabla_i\RMn_{0q}{}^q{}_{j}
   + \RMn_{iq}{}^{q}{}_{r}A^r_j
   + \RMn_{iqjr}A^{qr}
\end{align*}
or in $*$-notation
\begin{equation}
\Delta A = \nabla_{(2)}H + A*A*A + \nabla\Rcn(\nu,\cdot) + \Rcn*A.
\label{EQsimonsidentity}
\end{equation}

In most of our integral estimates, we include a function $\gamma:\Sigma\rightarrow\R$ in the integrand.
Eventually, this will be specialised to a smooth cutoff function between concentric geodesic balls on $\Sigma$.
For now however, let us only assume that $\gamma = \tilde{\gamma}\circ f$, where
\begin{equation}
\tag{$\gamma$}
0\le\tilde{\gamma}\le 1,\qquad\text{ and }\qquad
\vn{\tilde{\gamma}}_{C^2(N)} \le c_{\tilde{\gamma}} < \infty.
\end{equation}
Using the chain rule, this implies $D\gamma = (D\tilde{\gamma}\circ f)Df$ and then
$D^2\gamma = (D^2\tilde{\gamma}\circ f)(Df,Df) + (D\tilde{\gamma}\circ f)D^2f(\cdot,\cdot)$.
Thus there exists a constant $\cg=\cg(c_{\tilde{\gamma}}) \in \R$ such that 
\begin{equation}
\tag{$\gamma$}
\label{EQgamma}
|\nabla\gamma| \le \cg,\qquad\text{ and }\qquad
|\nabla_{(2)}\gamma| \le \cg(\cg+|A|).
\end{equation}
When we write ``for a function $\gamma:\Sigma\rightarrow\R$ as in \eqref{EQgamma}'' we mean a function
$\gamma:\Sigma\rightarrow\R$ as above, satisfying all conditions labeled \eqref{EQgamma}, which additionally
achieves the values zero and one in at least two points on $\Sigma$.

We note that if $\tilde{\gamma}$ is a cutoff function on a geodesic ball in $N$ of radius $\rho$, then we may
choose $\cg = \frac{c}\rho$ where $c$ is a universal constant and we have used that $c_{\tilde{\gamma}}=c_{\tilde{\gamma}}(\rho)$. This also explains our choice of scaling for the bound in \eqref{EQgamma}.


\section{Evolution equations for integrals of curvature.}

To begin, we state the following elementary evolution equations, whose proof is standard.

\begin{lem}
\label{LMevolutionequations}
For $f:\Sigma\times[0,T)\rightarrow N$ evolving by $\partial_tf = -\BF\nu$ the following equations hold:
\begin{align*}
  \pD{}{t}g_{ij} &= 2\BF A_{ij}\,,\qquad
  \pD{}{t}g^{ij}  = -2\BF A^{ij}\,,\\
  \pD{}{t}\nu &= \nabla \BF \,,\qquad
  \pD{}{t}d\mu  = H\BF \,d\mu\,,\\
  \pD{}{t}A_{ij} &= -\nabla_{ij}\BF  + A_{ik}A^k_j\BF  - \BF \RMn_{0i0j}\,.
\end{align*}
\end{lem}
Using $\BF = \BWv = \Delta H + H|A^o|^2 + H\Rcn(\nu,\nu)$, the $P$-notation introduced in the previous section, and the formula
\begin{align*}
\RMn_{ijkl} =
    \Rcn(e_i,e_l)\nIP{e_j}{e_k}
  - \Rcn(e_i,e_k)\nIP{e_j}{e_l}
 &- \Rcn(e_j,e_l)\nIP{e_i}{e_k}
  + \Rcn(e_j,e_k)\nIP{e_i}{e_l}
\\
 &- \frac12\Scn\Big(\nIP{e_i}{e_l}\nIP{e_j}{e_k} - \nIP{e_i}{e_k}\nIP{e_j}{e_l}\Big)\,,
\end{align*}
we write the evolution of the second fundamental form as
\begin{align*}
  \pD{}{t}A_{ij} = &- \nabla_{ij}\Delta H + \big(P_3^2 + P_5^0\big)(A) + \big(P_{1,1}^2 + P_{3,1}^0 + P_{1,2}^0+P_{2,1}^1\big)(A,\Rcn) \,.
\end{align*}
Interchanging covariant derivatives and applying \eqref{EQsimonsidentity} then gives the following lemma.

\begin{lem}
\label{LMevoforA}
For $f:\Sigma\times[0,T)\rightarrow N$ evolving by \eqref{EQwf} the following equation holds:
\begin{align*}
  \pD{}{t}A_{ij} = &- \Delta^2 A_{ij} + \big(P_3^2 + P_5^0 \big)(A) + \big(P_{1,1}^2 + P_{3,1}^0 + P_{1,2}^0 + P_{2,1}^1\big)(A,\Rcn)+ \Delta\nabla \Rcn\,.
\end{align*}
\end{lem}
Let us define $\SK_i$ to be the bound for the $i$-th ambient derivative of the Ricci curvature of $(N,\nIP{\cdot}{\cdot})$:
\begin{equation}
\label{EQambientcurvboundsdefn}
\SK_i :=
\sup_{p\in N,\,X,Y\in T_pN} \Big|\overline{\nabla_{(i)}\Rc}(X,Y)(p)\Big| < \infty\,.
\end{equation}
We now establish a basic energy estimate for \eqref{EQwf}.
In estimates such as these, integral quantities are evaluated at each immersion $f(\cdot,t)$ for
$t\in[0,T)$. This $t$-dependence is not typically noted, unless possibly ambiguous or when integrals
are evaluated at different times. This situation arises when integrating estimates such as Lemma
\ref{LMenergybasic} below.
\begin{lem}
Suppose $(N,\nIP{\cdot}{\cdot})$ is smooth and let $f:\Sigma\times[0,T)\rightarrow N$ be a solution of
\eqref{EQwf}.
For each $\delta>0$ there exists a constant $c\in(0,\infty)$ depending only on $\SK_i$ for $i\in\{0,1,2,3\}$, $\cg$ and $\delta$ such that the
following estimate holds:
\label{LMenergybasic}
\begin{align*}
\rD{}{t}&\int_\Sigma |A|^2\gamma^4d\mu
     + (2-\delta)\int_\Sigma |\nabla_{(2)}A|^2\gamma^4d\mu
\\&\le c\int_{[\gamma>0]}|A|^2d\mu
   + c\int_\Sigma|\overline{\nabla_{(3)}\Rc}|^2\,\gamma^4d\mu
   + c\int_\Sigma \big( |A|^6 + |\nabla A|^2|A|^2 + |\nabla A|^2\big)\gamma^4d\mu\,.
\end{align*}
\end{lem}
\begin{proof}
We first compute
\begin{align*}
-\int_\Sigma &\sIP{A}{\Delta^2A}\gamma^4\,d\mu
   =  -\int_\Sigma A^{ij}\nabla_p\Delta\nabla^pA_{ij}\gamma^4d\mu
   + \int_\Sigma A*\nabla(R*\nabla A)\gamma^4d\mu
\\&\le  -\int_\Sigma |\nabla_{(2)} A|^2\gamma^4d\mu
     -8\int_\Sigma A^{ij}(\nabla_q\nabla^p A_{ij})\big[(\nabla_p\nabla^q\gamma) \gamma
                                                       + 3(\nabla_p\gamma)(\nabla^q\gamma)\big]
                   \gamma^2d\mu
\\*&\quad
     -16\int_\Sigma (\nabla_pA^{ij})(\nabla_q\nabla^p A_{ij}) (\nabla^q\gamma) \gamma^3d\mu
\\*&\quad
       + \int_\Sigma \big(\nabla \Rcn * P_2^1(A) + \Rcn * P_2^2(A) + \Rcn * A * A * \nabla A\big)\gamma^4d\mu\,.
\end{align*}
We wish to now convert each derivative $\nabla$ acting on $\Rcn$ to $\nablan$
plus some error involving curvature.
A straightforward computation gives
\[
\nabla_{(k)}\Rcn
 = \overline{\nabla\Rc}
   + \sum_{l=0}^{k-1}\sum_{i+j=k-l} \overline{\nabla_{(l)}\Rc}*P_i^j(A)
\]
so that
\begin{align*}
|\nabla \Rcn|
 &\le c|\nablan\Rcn| + c|\Rcn|\,|A|\,,
\\
|\nabla_{(2)} \Rcn|
 &\le c|\overline{\nabla_{(2)}\Rc}| + c|\nablan\Rcn|\,|A|
    + c|\Rcn|\big(|A|^2 + |\nabla A|\big)\,,\text{ and}
\\
|\nabla_{(3)}\Rcn|
 &\le c|\overline{\nabla_{(3)}\Rc}| + c|\overline{\nabla_{(2)}\Rc}|\,|A|
    + c|\nablan\Rcn|\big(|A|^2 + |\nabla A|\big)
\\
 &\quad
    + c|\Rcn|\big(|A|^3 + |\nabla A|\,|A| + |\nabla_{(2)}A|\big)\,.
\end{align*}
For a differentiable function $\eta:\Sigma\rightarrow\R$ on $\Sigma$, the
evolution of $\int_\Sigma\eta\,d\mu$ under \eqref{EQwf} is
\[
\rD{}{t}\int_\Sigma \eta\,d\mu
 = \int_\Sigma \partial_t\eta\,d\mu
  + \int_\Sigma \eta H(\Delta H + H|A^o|^2 + H\Rcn(\nu,\nu))\,d\mu\,.
\]
Using the above formulae, Lemma \ref{LMevolutionequations} and Lemma \ref{LMevoforA} we compute
\begin{align*}
\rD{}{t}&\int_\Sigma |A|^2\gamma^4d\mu
\\*
 &= 2\int_\Sigma \sIP{A}{\partial_tA}\gamma^4d\mu
  + 2\int_\Sigma (\partial_tg^{ik})g^{jl}A_{ij}A_{kl}\gamma^4d\mu
  + 4\int_\Sigma |A|^2\gamma^3\partial_t\gamma\,d\mu
\\&\quad
  + \int_\Sigma H|A|^2(\Delta H + H|A^o|^2 + H\Rcn(\nu,\nu) )\gamma^4\,d\mu
\\
 &= 2\int_\Sigma \sIP{A}{- \Delta^2 A}\gamma^4d\mu
  + \int_\Sigma A*\Big(\big(P_3^2 + P_5^0 \big)(A)
                \Big)\gamma^4d\mu
\\&\quad
  + \int_\Sigma A*\Big(
                     \big(P_{1,2}^0+P_{2,1}^1 + P_{1,1}^2 + P_{3,1}^0\big)(A,\Rcn)
                   + \Delta\nabla \Rcn\Big)
                         \gamma^4d\mu
\\&\quad
  - 4\int_\Sigma (\Delta H + H|A^o|^2 + H\Rcn(\nu,\nu) )A^{ik}A_{ij}A_k^j\,\gamma^4d\mu
\\&\quad
  - 4\int_\Sigma |A|^2\gamma^3(D_\nu\tilde{\gamma}|_f)(\Delta H + H|A^o|^2 + H\Rcn(\nu,\nu))\,d\mu
\\&\quad
  + \int_\Sigma H|A|^2(\Delta H + H|A^o|^2 + H\Rcn(\nu,\nu))\gamma^4\,d\mu
\\
 &\le
     -2\int_\Sigma |\nabla_{(2)}A|^2\gamma^4d\mu
        + \delta\int_\Sigma |\nabla_{(2)} A|^2\gamma^4d\mu
\\&\quad
        + c(\cg)^2\int_\Sigma |\nabla A|^2\gamma^4d\mu
        + c\int_\Sigma |A|^6\gamma^4d\mu
        + c(\cg)^4\int_\Sigma |A|^2\gamma^4d\mu
\\&\quad
     -8\int_\Sigma A^{ij}(\nabla_q\nabla^p A_{ij})\big[(\nabla_p\nabla^q\gamma) \gamma
                                                       + 3(\nabla_p\gamma)(\nabla^q\gamma)\big]
                   \gamma^2d\mu
\\&\quad
     -16\int_\Sigma (\nabla_pA^{ij})(\nabla_q\nabla^p A_{ij}) (\nabla^q\gamma) \gamma^3d\mu
  + \int_\Sigma \big(P_4^2 + P_6^0 \big)(A)
                \gamma^4d\mu
\\&\quad
  + \int_\Sigma \big(\big(P_{3,1}^1 + P_{2,2}^0
                   + P_{2,1}^2 + P_{4,1}^0\big)(A,\Rcn)\big)\gamma^4d\mu
\\&\quad
  + c\int_\Sigma |A|\big(
      |\overline{\nabla_{(3)}\Rc}| + |\overline{\nabla_{(2)}\Rc}|\,|A|
    + |\nablan\Rcn|\big(|A|^2 + |\nabla A|\big)
\\
 &\qquad\qquad
    + |\Rcn|\big(|A|^3 + |\nabla A|\,|A| + |\nabla_{(2)}A|\big)\big)
                         \gamma^4d\mu
\\&\quad
  + \int_\Sigma \big(P_4^2(A) + P_6^0(A) + P_{4,1}^0(A,\Rcn)\big)\,\gamma^4d\mu
\\&\quad
  + c(\cg)\int_\Sigma \big(P_3^2(A) + P_5^0(A) + P_{3,1}^0(A,\Rcn) \big)\,\gamma^3d\mu
\\
 &\le
     -(2-\delta)\int_\Sigma |\nabla_{(2)}A|^2\gamma^4d\mu
     + c\big[(\cg)^4+(\cg)^2\big]\int_{[\gamma>0]}|A|^2d\mu
\\&\quad
     + c(\cg)^2\int_\Sigma \big(|A|^4 + |\nabla A|^2\big)\gamma^2d\mu
                    + c(\cg)\int_\Sigma \big(P_{3,1}^0(A,\Rcn) \big)\,\gamma^3d\mu
\\&\quad
  + \int_\Sigma \big(P_4^2 + P_6^0\big)(A)
                \gamma^4d\mu
  + c(\cg)\int_\Sigma \big(P_3^2 + P_5^0\big)(A)\,\gamma^3d\mu
\\&\quad
  + \int_\Sigma \big(\big(P_{3,1}^1 + P_{2,2}^0
                   + P_{2,1}^2 + P_{4,1}^0\big)(A,\Rcn)\big)\gamma^4d\mu
\\&\quad
  + c\int_\Sigma |A|\big(
      |\overline{\nabla_{(3)}\Rc}| + |\overline{\nabla_{(2)}\Rc}|\,|A|
    + |\nablan\Rcn|\big(|A|^2 + |\nabla A|\big)
\\
 &\qquad\qquad
    + |\Rcn|\big(|A|^3 + |\nabla A|\,|A| + |\nabla_{(2)}A|\big)\big)
                         \gamma^4d\mu
\\
 &\le
     -(2-\delta)\int_\Sigma |\nabla_{(2)}A|^2\gamma^4d\mu
     + c\big[(\cg)^4+(\cg)^2\big]\int_{[\gamma>0]}|A|^2d\mu
\\&\quad
     +c(\cg)^2\int_\Sigma |\nabla A|^2\gamma^2d\mu
     +c\int_\Sigma \big( |A|^6 + |\nabla A|^2|A|^2\big)\gamma^4d\mu
\\&\quad
  + c(\cg)\int_\Sigma \big(|\nabla_{(2)}A|\,|A|^2 + |\nabla A|^2|A| + |A|^5 + |A|^3|\Rcn|
                      \big)\,\gamma^3d\mu
\\&\quad
  + c \int_\Sigma \big(
                  |A|^2|\Rcn|^2 + |\nabla A|^2|\Rcn|
		+ |\nabla A|\,|A|\,|\nablan\Rcn|
		+ |A|^2|\overline{\nabla_{(2)}\Rc}|
\\&\qquad\qquad
                 + |A|^3|\nablan\Rcn|
                 + |\nabla A|\,|A|^2|\Rcn| + |A|^4|\Rcn|
    + |A|\,|\overline{\nabla_{(3)}\Rc}|               \big)\gamma^4d\mu
\\
 &\le
     -(2-\delta)\int_\Sigma |\nabla_{(2)}A|^2\gamma^4d\mu
     + c\big[(\cg)^4+(\cg)^2\big]\int_{[\gamma>0]}|A|^2d\mu
\\&\quad
     +c(\cg)^2\int_\Sigma |\nabla A|^2\gamma^2d\mu
     +c\int_\Sigma \big( |A|^6 + |\nabla A|^2|A|^2\big)\gamma^4d\mu
\\&\quad
  + c \int_\Sigma \big(
                  |A|^2|\Rcn|^2 + |\nabla A|^2|\Rcn|
		+ |\nabla A|\,|A|\,|\nablan\Rcn|
		+ |A|^2|\overline{\nabla_{(2)}\Rc}|
\\&\qquad\qquad
                 + |A|^3|\nablan\Rcn|
                 + |\nabla A|\,|A|^2|\Rcn| + |A|^4|\Rcn|
    + |A|\,|\overline{\nabla_{(3)}\Rc}|
                \big)\gamma^4d\mu
\\
 &\le
     -(2-\delta)\int_\Sigma |\nabla_{(2)}A|^2\gamma^4d\mu
     + c\big[(\cg)^4+1\big]\int_{[\gamma>0]}|A|^2d\mu
\\&\quad
     +c(\cg)^2\int_\Sigma |\nabla A|^2\gamma^2d\mu
     +c\int_\Sigma \big( |A|^6 + |\nabla A|^2|A|^2 + |\nabla A|^2 + |\overline{\nabla_{(3)}\Rc}|^2\big)\gamma^4d\mu\,.
\end{align*}
Estimating
\[
(\cg)^2\int_\Sigma |\nabla A|^2\gamma^2d\mu
 \le  \delta\int_\Sigma |\nabla_{(2)}A|^2\gamma^4d\mu + c(\cg)^4\int_{[\gamma>0]}|A|^2d\mu
\]
finishes the proof.
\end{proof}

\begin{lem}
\label{LMevolutionequationhigher}
For $f:\Sigma\times[0,T)\rightarrow N$ evolving by \eqref{EQwf} the following equation holds:
\begin{align*}
 \pD{}{t}\nabla_{(k)}A_{ij}
&= -\Delta^2\nabla_{(k)}A + \big(P_3^{k+2} + P_5^k\big)(A) + \overline{\nabla_{(3+k)}\Rc}
\\&\quad
  + \sum_{i+j+p=k+3} \Big(  \overline{\nabla_{(p)}\Rc} * P_j^i(A)\Big)
\\&\quad
  + \sum_{i+j+p=k+1} \Big(  \overline{\nabla_{(p)}(\Rc*\Rc)} * P_j^i(A)\Big)
\,.
\end{align*}
\end{lem}
\begin{proof}[Proof. (Sketch)]
We use an induction argument analogous to that found in \cite[Lemma 2.4]{kuwert2002gfw}.  Step one of induction for
$k=0$ can be easily obtained from the result of Lemma \ref{LMevoforA} and using the conversion from $\nabla$ to
$\nablan$ when acting on $\Rcn$ where the error terms involve the curvature and its derivatives, which we have already
computed in the proof of Lemma \ref{LMenergybasic}. To prove the induction step from $k$ to $k+1$ we use \cite[Lemma
2.3]{kuwert2002gfw} and again the conversion between the two types of derivative and absorb the extra terms arriving
from the different more complicated speed of the flow into the summations involving the Ricci curvature.
\end{proof}
\begin{lem}
\label{LMenergyhighervanillaconstantricci}
Suppose $(N,\nIP{\cdot}{\cdot})$ is smooth and let $f:\Sigma\times[0,T)\rightarrow N$ be a solution of \eqref{EQwf}.
For each $\delta>0$ there exists a constant $c\in(0,\infty)$ depending only on $\SK_i$ for $i=0,1,2$, $s$, $\cg$ and
$\delta$ such that the evolution of the concentration of $\nabla_{(k)}A$ in $L^2$ is estimated by
\begin{align*}
\rD{}{t}&\int_\Sigma|\nabla_{(k)}A|^2\gamma^sd\mu
    + (2-\delta)\int_\Sigma |\nabla_{(k+2)}A|^2 \gamma^sd\mu
\\*
&\le
    c\int_\Sigma |A|^2\gamma^{s-2k-4}d\mu
  + c\int_\Sigma \nabla_{(k)}A*\Big(P_3^{k+2}(A)+P_5^k(A)
  \Big)\,\gamma^sd\mu
\\&\hskip+1cm
  + c\int_\Sigma \nabla_{(k)}A*
    \bigg( \sum_{i+j+p=k+3} \Big(  \overline{\nabla_{(p)}\Rc} * P_j^i(A)\Big)
 \notag\\&\hskip+3cm
  + \sum_{i+j+p=k+1} \Big(  \overline{\nabla_{(p)}(\Rc*\Rc)} * P_j^i(A)\Big)
  \Big)\,\gamma^sd\mu
\,.
\end{align*}
\end{lem}
\begin{proof}
Integrating by parts and interchanging covariant derivatives, we obtain
\begin{align}
       2\int_\Sigma &\sIP{\nabla_{(k)}A}{
                                         - \Delta^2\nabla_{(k)}A
                }\gamma^sd\mu
\notag\\
 &=
      2\int_\Sigma \sIP{\nabla\nabla_{(k)}A}{
                                         \nabla\Delta\nabla_{(k)}A
                }\gamma^sd\mu
    +  \int_\Sigma \big(\nabla_{(k)}A*\nabla\Delta\nabla_{(k)}A*\nabla\gamma\big)
                \gamma^{s-1}d\mu
\notag\\
 &=
      2\int_\Sigma \sIP{\nabla\nabla_{(k)}A}{
                                         \nabla_p\nabla\nabla^p\nabla_{(k)}A
                }\gamma^sd\mu
\notag\\&\quad
    +  \int_\Sigma \big(\nabla_{(k+1)}A*\nabla_{(k+1)}A* (\Rcn+A*A)
                \big)\gamma^sd\mu
\notag\\&\quad
    +  \int_\Sigma \big(\nabla_{(k+2)}A*\nabla_{(k+1)}A*\nabla\gamma\big)
                \gamma^{s-1}d\mu
\notag\\&\quad
    +  \int_\Sigma \big(\nabla_{(k+2)}A*\nabla_{(k)}A*(\gamma\nabla_{(2)}\gamma + \nabla\gamma*\nabla\gamma\big)
                \gamma^{s-2}d\mu
\notag\\
 &=
    - 2\int_\Sigma |\nabla_{(k+2)}A|^2 \gamma^sd\mu
    + \int_\Sigma \big(\nabla_{(k+2)}A*\nabla_{(k+1)}A*\nabla\gamma\big)
                \gamma^{s-1}d\mu
\notag\\&\quad
    +  \int_\Sigma \big(\nabla_{(k+1)}A*\nabla_{(k+1)}A* (\Rcn+A*A)
                \big)\gamma^sd\mu
\notag\\&\quad
    +  \int_\Sigma \big(\nabla_{(k)}A*\nabla_{(k+2)}A* (\Rcn+A*A)
                \big)\gamma^sd\mu
\notag\\&\quad
    +  \int_\Sigma \big(\nabla_{(k+2)}A*\nabla_{(k+1)}A*\nabla\gamma\big)
                \gamma^{s-1}d\mu
\notag\\&\quad
    +  \int_\Sigma \big(\nabla_{(k+2)}A*\nabla_{(k)}A*(\gamma\nabla_{(2)}\gamma + \nabla\gamma*\nabla\gamma\big)
                \gamma^{s-2}d\mu
\notag\\
 &\le
    - (2-\delta)\int_\Sigma |\nabla_{(k+2)}A|^2 \gamma^sd\mu
    + c(\cg)^2\int_\Sigma |\nabla_{(k+1)}A|^2
                \gamma^{s-2}d\mu
\notag\\&\quad
    + c\int_\Sigma |\nabla_{(k)}A|^2|\Rcn|^2
                \gamma^sd\mu
    + c\int_\Sigma |\nabla_{(k)}A|^2|A|^4
                \gamma^sd\mu
\notag\\&\quad
    + c\int_\Sigma |\nabla_{(k+1)}A|^2|\Rcn|
                \gamma^sd\mu
    + c\int_\Sigma |\nabla_{(k+1)}A|^2|A|^2
                \gamma^sd\mu
\notag\\&\quad
    + c(\cg)^2\int_\Sigma |\nabla_{(k+2)}A|\,|\nabla_{(k)}A|\gamma^{s-2}d\mu
\notag\\&\quad
    + c(\cg)\int_\Sigma |\nabla_{(k+2)}A|\,|\nabla_{(k)}A|\,(\cg+|A|)\gamma^{s-1}d\mu
\notag\\
 &\le
    - (2-\delta)\int_\Sigma |\nabla_{(k+2)}A|^2 \gamma^sd\mu
    + c(\cg)^2\int_\Sigma |\nabla_{(k+1)}A|^2
                \gamma^{s-2}d\mu
\notag\\&\quad
    + c\int_\Sigma |\nabla_{(k)}A|^2|\Rcn|^2
                \gamma^sd\mu
    + c\int_\Sigma |\nabla_{(k)}A|^2|A|^4
                \gamma^sd\mu
\notag\\&\quad
    + c\int_\Sigma |\nabla_{(k+1)}A|^2|\Rcn|
                \gamma^sd\mu
\notag\\&\quad
    + c\int_\Sigma |\nabla_{(k+1)}A|^2|A|^2
                \gamma^sd\mu
    + c(\cg)^4\int_\Sigma |\nabla_{(k)}A|^2\gamma^{s-4}d\mu
\notag\\&\quad
    + c(\cg)^4\int_\Sigma |\nabla_{(k)}A|^2\gamma^{s-2}d\mu
    + c(\cg)^2\int_\Sigma |\nabla_{(k)}A|^2|A|^2\gamma^{s-2}d\mu\,.
\label{LMevolutionequationhigherEQderivation1}
\end{align}
A standard argument with integration by parts and Young's inequality yields the interpolation inequalities
\begin{equation}
\label{LMevolutionequationhigherEQint1}
\int_\Sigma |\nabla_{(k+1)}A|^2\gamma^{s-2}d\mu
 \le \delta\int_\Sigma |\nabla_{(k+2)}A|^2\gamma^sd\mu + c_\delta \int_\Sigma |A|^2\gamma^{s-2k-4}d\mu\,,
\end{equation}
and
\begin{align}
\int_\Sigma |\nabla_{(k)}A|^2\gamma^{s-4}d\mu
 &\le \int_\Sigma |\nabla_{(k+1)}A|^2\gamma^{s-2}d\mu + c \int_\Sigma |A|^2\gamma^{s-2k-4}d\mu
\notag\\
 &\le \delta\int_\Sigma |\nabla_{(k+2)}A|^2\gamma^sd\mu + c_\delta \int_\Sigma |A|^2\gamma^{s-2k-4}d\mu\,.
\label{LMevolutionequationhigherEQint2}
\end{align}
The estimates \eqref{LMevolutionequationhigherEQint1}, \eqref{LMevolutionequationhigherEQint2} above allow us
to interpolate away several terms on the right hand side of \eqref{LMevolutionequationhigherEQderivation1}.
For the remaining terms we use
\begin{align}
    c(\cg)^2&\int_\Sigma |\nabla_{(k+1)}A|^2|A|^2
                \gamma^{s}d\mu
   + c(\cg)^2\int_\Sigma |\nabla_{(k)}A|^2|A|^2\gamma^{s-2}d\mu
\notag\\&=
     c(\cg)^2\int_\Sigma |\nabla_{(k)}A|^2|A|^2\gamma^{s-2}d\mu
\notag\\&\qquad
   + c(\cg)^2\int_\Sigma \big(\nabla_{(k)}A*\nabla_{(k+2)}A*A*A\big)
                \gamma^{s-2}d\mu
\notag\\&\qquad
   + c(\cg)^2\int_\Sigma \big(\nabla_{(k)}A*\nabla_{(k+1)}A*\nabla A*A\big)
                \gamma^{s-2}d\mu
\notag\\&\qquad
   + (s-2)c(\cg)^2\int_\Sigma \big(\nabla_{(k)}A*\nabla_{(k+1)}A*A*A*\nabla\gamma\big)
                \gamma^{s-3}d\mu
\notag\\&\le
     c\int_\Sigma |\nabla_{(k+1)}A|^2\gamma^{s-2}d\mu
   + c\int_\Sigma |\nabla_{(k)}A|^2\gamma^{s-4}d\mu
\notag\\&\qquad
  + c\int_\Sigma \nabla_{(k)}A*\big(P_3^{k+2} + P_5^k\big)(A)\,\gamma^sd\mu
\label{LMevolutionequationhigherEQestimate1}
\end{align}
In the last estimate above and for the rest of the proof we shall no longer write the dependence of $c$ on
$\cg$ explicitly.
Combining \eqref{LMevolutionequationhigherEQderivation1}, \eqref{LMevolutionequationhigherEQint1},
\eqref{LMevolutionequationhigherEQint2}, and \eqref{LMevolutionequationhigherEQestimate1} we obtain the
estimate
\begin{align}
       2\int_\Sigma &\sIP{\nabla_{(k)}A}{
                                         - \Delta^2\nabla_{(k)}A
                }\gamma^sd\mu
    + (2-\delta)\int_\Sigma |\nabla_{(k+2)}A|^2 \gamma^sd\mu
\notag\\
 &\le
    c\int_\Sigma \nabla_{(k)}A*\big(P_3^{k+2} + P_5^k\big)(A)\,\gamma^sd\mu
  + c\int_\Sigma |A|^2\gamma^{s-2k-4}d\mu\,,
\label{LMevolutionequationhigherEQestimate2}
\end{align}
where $c$ depends only on $\SK_0$, $s$, $\cg$ and $\delta$.
The estimate \eqref{LMevolutionequationhigherEQestimate2} is the leading order term in our main estimate below
for the evolution of the concentration of curvature in $L^2$.

We compute
\begin{align}
\rD{}{t}&\int_\Sigma|\nabla_{(k)}A|^2\gamma^sd\mu
\notag\\&=
       2\int_\Sigma \sIP{\nabla_{(k)}A}{\partial_t\nabla_{(k)}A}\gamma^sd\mu
\notag\\&\quad
         + 2\int_\Sigma \big(\nabla_{(k)}A*\nabla_{(k)}A*A\big)
                    \big(\Delta H + H|A^o|^2 + H\Rcn(\nu,\nu)\big)\gamma^sd\mu
\notag\\&\quad
     + s\int_\Sigma |\nabla_{(k)}A|^2\partial_t\gamma\gamma^{s-1}d\mu
    \notag\\&\quad
     + 2\int_\Sigma |\nabla_{(k)}A|^2H\big(
                        \Delta H + H|A^o|^2 + H\Rcn(\nu,\nu)
                        \big)\gamma^{s}d\mu
\notag\\&=
       2\int_\Sigma \sIP{\nabla_{(k)}A}{
                                         - \Delta^2\nabla_{(k)}A
                }\gamma^sd\mu
\notag\\&\quad
     + \int_\Sigma \nabla_{(k)}A*
                       \big(P_3^{k+2} + P_5^k\big)(A)\,\gamma^sd\mu
\notag\\&\quad
+\int_\Sigma \nabla_{(k)}A*
                       \bigg(
                     \sum_{i+j+p=k+3} \Big(  \overline{\nabla_{(p)}\Rc} * P_j^i(A)\Big)
\notag\\&\hskip+3cm
  + \sum_{i+j+p=k+1} \Big(  \overline{\nabla_{(p)}(\Rc*\Rc)} * P_j^i(A)\Big)\bigg)\gamma^sd\mu
\notag\\&\quad
     + 2\int_\Sigma \big(\nabla_{(k)}A*\nabla_{(k)}A*A\big)
                    \big(\Delta H + H|A^o|^2 + H\Rcn(\nu,\nu)\big)
                    \gamma^sd\mu
\notag\\&\quad
    + s\int_\Sigma |\nabla_{(k)}A|^2
 (D_\nu\tilde{\gamma}|_f)(\Delta H + H|A^o|^2 + H\Rcn(\nu,\nu))
                   \gamma^{s-1}d\mu
\notag\\&\quad
    + 2\int_\Sigma |\nabla_{(k)}A|^2H\big(
                        \Delta H + H|A^o|^2 + H\Rcn(\nu,\nu)
                        \big)\gamma^{s}d\mu
\notag\\&\le
    - (2-\delta)\int_\Sigma |\nabla_{(k+2)}A|^2 \gamma^sd\mu
  + c\int_\Sigma \nabla_{(k)}A*\big(P_3^{k+2} + P_5^k\big)(A)\,\gamma^sd\mu
\notag\\&\quad
  + c\int_\Sigma |A|^2\gamma^{s-2k-4}d\mu
\notag\\&\quad
     +\int_\Sigma \nabla_{(k)}A*\bigg(
                                \sum_{i+j+p=k+3} \Big(  \overline{\nabla_{(p)}\Rc} * P_j^i(A)\Big)
\notag\\&\hskip+3cm
+ \sum_{i+j+p=k+1} \Big(  \overline{\nabla_{(p)}(\Rc*\Rc)} * P_j^i(A)\Big)\bigg)\gamma^sd\mu
\notag\\&\quad
  + c\int_\Sigma \nabla_{(k)}A*\nabla_{(k)}A*A*A*\Rcn\,\gamma^sd\mu
\notag\\&\quad
    + c\int_\Sigma |\nabla_{(k)}A|^2
 (D_\nu\tilde{\gamma}|_f)(\Delta H + H|A^o|^2 + H\Rcn(\nu,\nu))
                   \gamma^{s-1}d\mu
\notag\\&\quad
     + c\int_\Sigma |\nabla_{(k)}A|^2H\big(
                        \Delta H + H|A^o|^2 + H\Rcn(\nu,\nu)
                        \big)\gamma^{s}d\mu\,,
\label{LMevolutionequationhigherEQderivation2}
\end{align}
where \eqref{LMevolutionequationhigherEQestimate2} was used
in the last step.
The rightmost triple of integrals are estimated with
\begin{align}
    c\int_\Sigma &|\nabla_{(k)}A|^2
 (D_\nu\tilde{\gamma}|_f)(\Delta H + H|A^o|^2 + H\Rcn(\nu,\nu))
                   \gamma^{s-1}d\mu
\notag\\
  &+ c\int_\Sigma |\nabla_{(k)}A|^2H\big(
                        \Delta H + H|A^o|^2 + H\Rcn(\nu,\nu)
                        \big)\gamma^{s}d\mu
\notag\\
  &+ c\int_\Sigma \nabla_{(k)}A*\nabla_{(k)}A*A*A*\Rcn\,\gamma^sd\mu
\notag\\
&\le
   -c\int_\Sigma  \nabla H*\nabla\big(|\nabla_{(k)}A|^2
 (D_\nu\tilde{\gamma}|_f)
                   \gamma^{s-1}\big)d\mu
\notag\\&\quad
   + c\int_\Sigma |\nabla_{(k)}A|^2\gamma^{s}d\mu
   + c\int_\Sigma |\nabla_{(k)}A|^2|A|^2\gamma^{s}d\mu
   + c\int_\Sigma |\nabla_{(k)}A|^2|A|^2\gamma^{s-2}d\mu
\notag\\&\quad
   + c\int_\Sigma \nabla_{(k)}A*\big(P_3^{k+2} + P_5^k\big)(A)\,\gamma^sd\mu
\notag\\
&\le
    c\int_\Sigma  |\nabla A|\,|\nabla_{(k)}A|\,|\nabla_{(k+1)}A|\,
                   \gamma^{s-1}d\mu
  + c\int_\Sigma  |\nabla A|\,|\nabla_{(k)}A|^2\,(1+|A|)\,
                   \gamma^{s-1}d\mu
\notag\\&\quad
  + c\int_\Sigma  |\nabla A|\,|\nabla_{(k)}A|^2
                   \gamma^{s-2}d\mu
   + c\int_\Sigma |\nabla_{(k)}A|^2\gamma^{s-4}d\mu
\notag\\&\quad
   + c\int_\Sigma |\nabla_{(k)}A|^2|A|^2\gamma^{s-2}d\mu
   + c\int_\Sigma \nabla_{(k)}A*\big(P_3^{k+2} + P_5^k\big)(A)\,\gamma^sd\mu
\notag\\
&\le
    c\int_\Sigma  |\nabla_{(k+1)}A|^2
                   \gamma^{s-2}d\mu
  + c\int_\Sigma  |\nabla_{(k)}A|^2
                   \gamma^{s-4}d\mu
  + c\int_\Sigma |\nabla_{(k)}A|^2|A|^2\gamma^{s-2}d\mu
\notag\\&\quad
  + c\int_\Sigma  |\nabla_{(k)}A|^2|\nabla A|^2
                   \gamma^{s}d\mu
   + c\int_\Sigma \nabla_{(k)}A*\big(P_3^{k+2} + P_5^k\big)(A)\,\gamma^sd\mu
\notag\\
&\le
    c\int_\Sigma  |\nabla_{(k+1)}A|^2
                   \gamma^{s-2}d\mu
  + c\int_\Sigma  |\nabla_{(k)}A|^2
                   \gamma^{s-4}d\mu
  + c\int_\Sigma |\nabla_{(k)}A|^2|A|^2\gamma^{s-2}d\mu
\notag\\&\quad
   + c\int_\Sigma \nabla_{(k)}A*\big(P_3^{k+2} + P_5^k\big)(A)\,\gamma^sd\mu
\notag\\&\le
     \delta\int_\Sigma |\nabla_{(k+2)}A|^2\gamma^{s}d\mu
   + c\int_\Sigma |\nabla_{(k+1)}A|^2\gamma^{s-2}d\mu
   + c\int_\Sigma |\nabla_{(k)}A|^2\gamma^{s-4}d\mu
\notag\\&\qquad
  + c\int_\Sigma \nabla_{(k)}A*\big(P_3^{k+2} + P_5^k\big)(A)\,\gamma^sd\mu
\label{LMevolutionequationhigherEQestimate3}
\,,
\end{align}
where we again applied \eqref{LMevolutionequationhigherEQestimate1}.
Combining \eqref{LMevolutionequationhigherEQestimate3} with \eqref{LMevolutionequationhigherEQderivation2} and using \eqref{LMevolutionequationhigherEQint1},
\eqref{LMevolutionequationhigherEQint2}, we have
\begin{align}
\rD{}{t}&\int_\Sigma|\nabla_{(k)}A|^2\gamma^sd\mu
    + (2-\delta)\int_\Sigma |\nabla_{(k+2)}A|^2 \gamma^sd\mu
\notag\\
&\le
    c\int_\Sigma \nabla_{(k)}A*\big(P_3^{k+2} + P_5^k\big)(A)\,\gamma^sd\mu
  + c\int_\Sigma |A|^2\gamma^{s-2k-4}d\mu
\notag\\&\quad
     + \int_\Sigma \nabla_{(k)}A*
                       \bigg( \sum_{i+j+p=k+3} \Big(  \overline{\nabla_{(p)}\Rc} * P_j^i(A)\Big)
 \notag\\&\hskip+3cm
  + \sum_{i+j+p=k+1} \Big(  \overline{\nabla_{(p)}(\Rc*\Rc)} * P_j^i(A)\Big)
                           \bigg)\gamma^sd\mu
\,,
\label{LMevolutionequationhigherEQderivation4}
\end{align}
This finishes the proof.
\end{proof}

For later application we split out the cases $k=1$ and $k=2$ from Lemma \ref{LMenergyhighervanillaconstantricci}.

\begin{lem}
\label{LMenergygrad1}
Suppose $(N,\nIP{\cdot}{\cdot})$ is smooth and let $f:\Sigma\times[0,T)\rightarrow N$ be a solution of \eqref{EQwf}.
There exists a constant $c\in(0,\infty)$ depending only on $\SK_i$ for $i=0,1,2,3$, $s$ and $\cg$ such that the
evolution of the concentration of $\nabla A$ in $L^2$ is estimated by
\begin{align*}
\rD{}{t}&\int_\Sigma|\nabla A|^2\gamma^sd\mu
    + \frac32\int_\Sigma |\nabla_{(3)}A|^2 \gamma^sd\mu
\\*
&\le
    c\int_\Sigma |A|^2\gamma^{s-6}d\mu
  + c\int_\Sigma |\overline{\nabla_{(4)}\Rc}|^2\gamma^sd\mu
  + c\int_\Sigma |A|^6\gamma^{s}d\mu
\\&\hskip+2.86cm
  + c\int_\Sigma \nabla A*\big(P_3^3 + P_5^1\big)(A)\,\gamma^sd\mu
\,.
\end{align*}
\end{lem}
\begin{proof}
From Lemma \ref{LMenergyhighervanillaconstantricci} with $k=1$ we find
\begin{align*}
\rD{}{t}&\int_\Sigma|\nabla A|^2\gamma^sd\mu
    + \frac32\int_\Sigma |\nabla_{(3)}A|^2 \gamma^sd\mu
\\*
&\le
    c\int_\Sigma |A|^2\gamma^{s-6}d\mu
  + c\int_\Sigma \nabla A*\big(P_3^3 + P_5^1\big)(A)\,\gamma^sd\mu
\\&\hskip+2cm
  + c\int_\Sigma \nabla A*\big(\overline{\nabla_{(4)}\Rc} + \overline{\nabla_{(3)}\Rc}*A\big)\,\gamma^sd\mu
\\&\hskip+2cm
  + c\int_\Sigma \nabla A*\big(|P_1^0|+|P_1^1|+|P^2_1|+|P^3_1|+|P^0_2|
  \\&\qquad\qquad\qquad\qquad\qquad\qquad
  +|P_2^1|+|P^2_2|+|P_3^0|+|P^1_3|+|P_4^0|\big)(A)\,\gamma^sd\mu
\,.
\end{align*}
The constant on the right hand side depends on $\SK_i$ for $i=0,1,2$.
From this point onward we also allow the constant $c$ to depend upon $\SK_3$.
We have used the notation $|P_i^j|$ to mean that the norm of each term in that sum of contractions is taken.
Using Young's inequality we estimate the right hand side by
\begin{align}
\notag
c\int_\Sigma &\nabla A*\big(\overline{\nabla_{(4)}\Rc} + P_3^3 + |P_1^0|+|P_1^1|+|P^2_1|+|P^3_1|+|P^0_2|
\\&\qquad\qquad\qquad\qquad
  +|P_2^1|+|P^2_2|+|P_3^0|+|P^1_3|+|P_4^0|+P_5^1\big)(A)\,\gamma^sd\mu
\notag\\
&\le
   c\int_\Sigma \nabla A*\big(\overline{\nabla_{(4)}\Rc} + P_3^3 + P_5^1 +
    |A| + |\nabla A| + |\nabla_{(2)}A| + |\nabla_{(3)}A|
\notag
  \\&\qquad\qquad\qquad\qquad
   + |A*\nabla A| + |\nabla A| * |\nabla A| + |A*\nabla_{(2)}A| + |A*A*A|
\notag
  \\&\qquad\qquad\qquad\qquad\qquad
    + |A*A| + |A*A*\nabla A| + |A*A*A*A|\big)(A)\,\gamma^sd\mu
\notag\\
&\le
   c\int_\Sigma \nabla A*\big(P_3^3 + P_5^1\big)(A)\,\gamma^sd\mu
 + \delta\int_\Sigma |\nabla_{(3)} A|^2\gamma^sd\mu
 + c\int_\Sigma |\nabla_{(2)} A|^2\gamma^sd\mu
\notag\\&
\qquad
 + c\int_\Sigma |\nabla A|^2\gamma^sd\mu
 + c\int_\Sigma |\nabla A|^4\gamma^sd\mu
 + c\int_\Sigma |A|^6\gamma^sd\mu
\notag\\&
\qquad
 + c\int_\Sigma |A|^2\gamma^sd\mu
 + c\int_\Sigma |\overline{\nabla_{(4)}\Rc}|^2\gamma^sd\mu
\,.
\label{LMenergygrad1EQ1}
\end{align}
Let us now estimate each of the terms on the right hand side of \eqref{LMenergygrad1EQ1} in turn.
We begin with
\begin{align}
 \int_\Sigma |\nabla_{(2)} A|^2\gamma^sd\mu
  &= - \int_\Sigma \sIP{\nabla_{(3)} A}{\nabla A}\,\gamma^sd\mu
    - s\int_\Sigma \sIP{\nabla_{(2)} A}{\nabla A\nabla\gamma}\,\gamma^{s-1}d\mu
\notag\\
 &\le
      \delta\int_\Sigma |\nabla_{(3)} A|^2\gamma^sd\mu
    + \frac12\int_\Sigma |\nabla_{(2)} A|^2\gamma^sd\mu
\notag\\
 &\qquad
    + c\int_\Sigma |\nabla A|^2\gamma^sd\mu
    + c\int_\Sigma |\nabla A|^2\gamma^{s-2}d\mu
\notag
\intertext{so}
 \int_\Sigma |\nabla_{(2)} A|^2\gamma^sd\mu
  &\le
      \delta\int_\Sigma |\nabla_{(3)} A|^2\gamma^sd\mu
    + c\int_\Sigma |\nabla A|^2\gamma^{s-2}d\mu\,.
\label{LMenergygrad1EQ11}
\intertext{An analogous computation yields}
 \int_\Sigma |\nabla A|^2\gamma^{s-2}d\mu
  &\le
      \delta\int_\Sigma |\nabla_{(2)} A|^2\gamma^sd\mu
    + c\int_\Sigma |A|^2\gamma^{s-4}d\mu\,.
\label{LMenergygrad1EQ2}
\end{align}
Combining \eqref{LMenergygrad1EQ11} and \eqref{LMenergygrad1EQ2} we find
\begin{align}
 \int_\Sigma |\nabla_{(2)} A|^2\gamma^{s}d\mu
  &\le
      \delta\int_\Sigma |\nabla_{(3)} A|^2\gamma^sd\mu
    + c\int_\Sigma |A|^2\gamma^{s-4}d\mu\,.
\label{LMenergygrad1EQ3}
\end{align}
Estimate \eqref{LMenergygrad1EQ3} deals with the third term on the right hand side of \eqref{LMenergygrad1EQ1}, while the fourth term on the right hand side of
\eqref{LMenergygrad1EQ1} is estimated by first applying \eqref{LMenergygrad1EQ2} and then using \eqref{LMenergygrad1EQ3}.
Observe that the fifth term is already no problem:
\[
 \int_\Sigma |\nabla A|^4\gamma^sd\mu
 = \int_\Sigma \nabla A*P_3^3(A)\,\gamma^sd\mu\,.
\]
Choosing $\delta$ sufficiently small and absorbing finishes the proof.
\end{proof}

\begin{lem}
\label{LMenergygrad2}
Suppose $(N,\nIP{\cdot}{\cdot})$ is smooth and let $f:\Sigma\times[0,T)\rightarrow N$ be a solution of \eqref{EQwf}.
There exists a constant $c\in(0,\infty)$ depending only on $\SK_i$ for $i=0,1,2,3,4$, $s$ and $\cg$ such that the
evolution of the concentration of $\nabla_{(2)}A$ in $L^2$ is estimated by
\begin{align*}
\rD{}{t}&\int_\Sigma|\nabla_{(2)}A|^2\gamma^sd\mu
    + \frac32\int_\Sigma |\nabla_{(4)}A|^2 \gamma^sd\mu
\notag
\\*
&\le
   c\int_\Sigma \nabla_{(2)} A*\big(P_3^4 + P_5^2\big)(A)\,\gamma^sd\mu
 + c\int_\Sigma |A|^6\gamma^sd\mu
\notag\\&\qquad
 + c\int_\Sigma |\overline{\nabla_{(5)}\Rc}|^2\gamma^sd\mu
 + c\int_\Sigma |A|^2\gamma^{s-8}d\mu
\,.
\end{align*}
\end{lem}
\begin{proof}
From Lemma \ref{LMenergyhighervanillaconstantricci} with $k=2$ we find
\begin{align}
\rD{}{t}\int_\Sigma&|\nabla_{(2)}A|^2\gamma^sd\mu
    + (2-\delta)\int_\Sigma |\nabla_{(4)}A|^2 \gamma^sd\mu
\notag
\\
&\le
    c\int_\Sigma |A|^2\gamma^{s-8}d\mu
\notag\\&\hskip+1cm
  + c\int_\Sigma \nabla_{(2)}A*\big(\overline{\nabla_{(5)}\Rc}
                                  + \overline{\nabla_{(4)}\Rc}*A
\notag \\&\hskip+2cm
                                  + \overline{\nabla_{(3)}\Rc}*A*A
                                  + \overline{\nabla_{(3)}\Rc}*\nabla A
                               \big)\,\gamma^sd\mu
\notag\\&\hskip+1cm
  + c\int_\Sigma \nabla_{(2)}A*\big(P_3^{4} + P_5^2 + |P_1^0|+|P_1^1|
\notag \\&\hskip+2cm
  + |P^2_1| + |P^3_1| + |P_1^4| + |P^0_2| + |P_2^1| + |P^2_2| + |P_2^3|
\notag \\&\hskip+2cm
  + |P_3^0| + |P^1_3| + |P_3^2| + |P_4^0| + |P_4^1| + |P^0_5|\big)(A)\,\gamma^sd\mu
\notag
\\
&\le
    c\int_\Sigma |A|^2\gamma^{s-8}d\mu
  + c\int_\Sigma |\overline{\nabla_{(5)}\Rc}|^2\gamma^sd\mu
\notag\\&\hskip+1cm
  + c\int_\Sigma \nabla_{(2)}A*\big(P_3^{4} + P_5^2 + |P_1^0|+|P_1^1|
\notag \\&\hskip+2cm
  + |P^2_1| + |P^3_1| + |P_1^4| + |P^0_2| + |P_2^1| + |P^2_2| + |P_2^3|
\notag \\&\hskip+2cm
  + |P_3^0| + |P^1_3| + |P_3^2| + |P_4^0| + |P_4^1| + |P^0_5|\big)(A)\,\gamma^sd\mu
\,.
\label{LMenergygrad2EQ0}
\end{align}
As in Lemma \ref{LMenergygrad1} above we use Young's inequality to estimate
\begin{align}
\notag
  c\int_\Sigma &\nabla_{(2)}A*\big(P_3^{4} + P_5^2 + |P_1^0|+|P_1^1|+|P^2_1|+|P^3_1|+|P_1^4|
\\&\quad
 +|P^0_2| + |P_2^1|+|P^2_2|+|P_2^3|+|P_3^0|+|P^1_3|+|P_3^2|
\notag\\&\quad\hskip+2cm
 +|P_4^0|+|P_4^1|+|P^0_5|\big)(A)\,\gamma^sd\mu
\notag\\
&\le
   c\int_\Sigma \nabla_{(2)} A*\big( P_3^4 + P_5^2 +
     |A| + |\nabla A| + |\nabla_{(2)}A| + |\nabla_{(3)}A| + |\nabla_{(4)}A|
\notag\\&\qquad\qquad
   + |A*A| + |A*\nabla A| + |\nabla A| * |\nabla A| + |A*\nabla_{(2)}A|
\notag \\&\qquad\qquad
   + |\nabla A|*|\nabla_{(2)}A|
   + |A*\nabla_{(3)}A|
   + |A*A*\nabla A|
   + |A*A*A|
\notag \\&\qquad\qquad
   + |A*A*\nabla_{(2)}A|
   + |A* \nabla A * \nabla A|
   + |A*A*A*A|
\notag \\&\qquad\qquad
   + |A*A*A*\nabla A| + |A*A*A*A*A|\big)(A)\,\gamma^sd\mu
\notag\\
&\le
   c\int_\Sigma \nabla_{(2)} A*\big(P_3^4 + P_5^2\big)(A)\,\gamma^sd\mu
 + \delta\int_\Sigma |\nabla_{(4)} A|^2\gamma^sd\mu
\notag\\&\qquad
 + c\int_\Sigma |\nabla_{(3)} A|^2\gamma^sd\mu
 + c\int_\Sigma |\nabla_{(2)} A|^2\gamma^sd\mu
 + c\int_\Sigma |\nabla A|^2\gamma^sd\mu
\notag\\&\qquad
 + c\int_\Sigma |\nabla_{(2)} A|^2\big(|A|^4 + |\nabla A|^2\big)\,\gamma^sd\mu
 + c\int_\Sigma |\nabla A|^4\gamma^sd\mu
\notag\\&\qquad
 + c\int_\Sigma |A|^2\gamma^sd\mu
 + c\int_\Sigma |A|^6\gamma^sd\mu
\,.
\label{LMenergygrad2EQ1}
\end{align}
As earlier we estimate
\begin{align}
 \int_\Sigma |\nabla_{(3)} A|^2\gamma^{s}d\mu
  &\le
      \delta\int_\Sigma |\nabla_{(4)} A|^2\gamma^sd\mu
    + c\int_\Sigma |A|^2\gamma^{s-6}d\mu
\,.
\label{LMenergygrad2EQ2}
\end{align}
Estimate \eqref{LMenergygrad2EQ2} deals with the third term on the right hand side of \eqref{LMenergygrad2EQ1}, while the fourth term on the right hand side of
\eqref{LMenergygrad2EQ1} is estimated by first applying \eqref{LMenergygrad1EQ3} and then using \eqref{LMenergygrad2EQ2}.
The fifth term is estimated by applying \eqref{LMenergygrad1EQ2}, then \eqref{LMenergygrad1EQ3} and finally \eqref{LMenergygrad2EQ2}.
Observe that the sixth term is of the form
\[
 \int_\Sigma |\nabla_{(2)} A|^2\big(|A|^4 + |\nabla A|^2\big)\,\gamma^sd\mu
 = \int_\Sigma \nabla_{(2)}A*\big(P_3^4+P_5^2\big)(A)\,\gamma^sd\mu\,.
\]
Choosing $\delta$ sufficiently small and absorbing in \eqref{LMenergygrad2EQ0} we find
\begin{align}
\rD{}{t}&\int_\Sigma|\nabla_{(2)}A|^2\gamma^sd\mu
    + (2-\delta)\int_\Sigma |\nabla_{(4)}A|^2 \gamma^sd\mu
\notag
\\*
&\le
   c\int_\Sigma \nabla_{(2)} A*\big(P_3^4 + P_5^2\big)(A)\,\gamma^sd\mu
 + c\int_\Sigma |A|^6\gamma^sd\mu
 + c\int_\Sigma |A|^2\gamma^{s-8}d\mu
\notag\\&\qquad
 + c\int_\Sigma |\nabla A|^4\gamma^sd\mu
\,.
\label{LMenergygrad2EQ3}
\end{align}
For the last term we integrate by parts and estimate using Young's inequality repeatedly to obtain
\begin{align*}
\int_\Sigma|\nabla A|^4\gamma^sd\mu
 &= \int_\Sigma \nabla_{(2)}A*A*\nabla A*\nabla A\,\gamma^sd\mu
  + \int_\Sigma \nabla \gamma * A * \nabla A * \nabla A\,\gamma^{s-1}d\mu
\notag\\&\le
   \frac14\int_\Sigma|\nabla A|^4\gamma^sd\mu
  + c\int_\Sigma |\nabla_{(2)}A|^2(1+|A|^4)\gamma^sd\mu
\notag\\&\qquad
  + \int_\Sigma \nabla \gamma * A * \nabla A * \nabla A * \nabla A\,\gamma^{s-1}d\mu
\notag\\&\le
   \frac12\int_\Sigma|\nabla A|^4\gamma^sd\mu
  + c\int_\Sigma |\nabla_{(2)}A|^2(1+|A|^4)\gamma^sd\mu
\notag\\&\qquad
  + c\int_\Sigma |\nabla A|^2|A|^2\,\gamma^{s-2}d\mu
\notag\\&\le
   \frac34\int_\Sigma|\nabla A|^4\gamma^sd\mu
  + c\int_\Sigma |\nabla_{(2)}A|^2(1+|A|^4)\gamma^sd\mu
\notag\\&\qquad
  + c\int_\Sigma |A|^4\,\gamma^{s-4}d\mu
\notag\\&\le
   \frac34\int_\Sigma|\nabla A|^4\gamma^sd\mu
  + c\int_\Sigma |\nabla_{(2)}A|^2(1+|A|^4)\gamma^sd\mu
\notag\\&\qquad
  + c\int_\Sigma |A|^6\,\gamma^{s}d\mu
  + c\int_\Sigma |A|^2\,\gamma^{s-8}d\mu\,.
\end{align*}
Absorbing yields
\begin{align}
\int_\Sigma|\nabla A|^4\gamma^sd\mu
 &\le c\int_\Sigma |\nabla_{(2)}A|^2(1+|A|^4)\gamma^sd\mu
\notag\\&\qquad
  + c\int_\Sigma |A|^6\,\gamma^{s}d\mu
  + c\int_\Sigma |A|^2\,\gamma^{s-8}d\mu\,.
\label{LMenergygrad2EQ4}
\end{align}
Combining \eqref{LMenergygrad2EQ4} with \eqref{LMenergygrad2EQ3} and absorbing with estimates \eqref{LMenergygrad1EQ3} and then \eqref{LMenergygrad2EQ2} finishes the proof.
\end{proof}

To deal with the sixth term on the right hand side of \eqref{LMenergygrad1EQ1} we have to apply the Hoffman-Spruck
Sobolev inequality, which takes various forms depending on the geometry of $(N,\nIP{\cdot}{\cdot})$. This is done in the
next section.


\section{Integral estimates with small concentration of curvature.}

We will primarily use the Hoffman-Spruck Sobolev inequality \cite{hoffmanspruck}, which is the famous Michael-Simon
Sobolev inequality \cite{michael1973sam} adapted to submanifolds of Riemannian spaces.
We state here a version of \cite[Theorem 2.1]{hoffmanspruck} which is tailored to our situation:

\begin{thm}[Hoffman-Spruck Sobolev inequality for solutions of \eqref{EQwf}]
\label{TMhs}
Let $f:\Sigma\times[0,T)\rightarrow N$ be a family of immersed surfaces and $u\in C_c^1(\Sigma\times[0,T))$.
Assume that
\begin{equation}
\SK :=
\sup_{p\in N,\,X,Y\in T_pN} \Rcn(X,Y)(p)
\le \frac{8\pi}{9|\Sigma|_{[u>0]}}
\,,
\label{EQambientcurvass}
\end{equation}
and
\begin{equation}
\rho_N \ge
\begin{cases}
\frac{1}{2\sqrt\SK} \sin^{-1}\Big( \sqrt{9\SK|\Sigma|_{[u>0]}/4\pi} \Big)
\,, \text{ if $\SK\ge0$,}
\\
\frac34\sqrt{|\Sigma|_{[u>0]}/\pi}
\qquad\qquad\ \quad\qquad \text{ if $\SK<0$,}
\end{cases}
\label{EQambientinjass}
\end{equation}
where $\rho_N$ is the injectivity radius of $N$.
Then we have
\[
  \left(\int_\Sigma |u|^2d\mu\right)^{1/2}
    \le \frac{9\sqrt{\pi}}{2} \int_\Sigma |\nabla u| + |u|\ |H|\, d\mu\,.
\]
\end{thm}

\begin{rmk}
If $N$ is simply connected, complete, and with non-positive sectional curvature, the condition \eqref{EQambientcurvass}
is automatically satisfied regardless of the value of $|\Sigma|$ along the flow.  Furthermore, by the Cartan-Hadamard
theorem the injectivity radius $\rho_N = \infty$ on such a manifold, and so \eqref{EQambientinjass} is also satisfied.
This means that the assumption \eqref{EQambeintassumptionsforlifespan} implies that we may apply Theorem \ref{TMhs}
along the flow for any differentiable test function $u$ with compact support at any time.
\end{rmk}

A straightforward consequence of Theorem \ref{TMhs} is the following multiplicative Sobolev inequality.
With Theorem \ref{TMhs} in hand, the proof follows the same argument as in \cite[Lemma 4.2]{kuwert2002gfw}, and so we omit it.

\begin{lem}
\label{LMmultsob}
Let $f:\Sigma\times[0,T)\rightarrow N$ be a family of immersions and let $\gamma$ be as in \eqref{EQgamma} satisfying
the assumptions \eqref{EQambientcurvass}, \eqref{EQambientinjass} of Theorem \ref{TMhs} with $u=\gamma$.  If $s\ge4$
then we have
\begin{align*}
\int_\Sigma \big(|\nabla A|^2|A|^2 + |A|^6\big)\gamma^sd\mu
   &\le  c\int_{[\gamma>0]}|A|^2d\mu\int_\Sigma\big(|\nabla_{(2)}A|^2 + |A|^6\big)\gamma^sd\mu
   \\*
   &\hskip+2cm + c(\cg)^4\Big( \int_{[\gamma>0]}|A|^2d\mu \Big)^2,
\end{align*}
where $c$ is a constant depending only on $s$.
\end{lem}

For solutions in simply-connected 3-manifolds with non-positive Ricci curvature, we immediately obtain our desired control of the concentration of curvature
along the flow.

\begin{prop}[Control on the concentration of curvature]
\label{PRconcentestvanilla}
Suppose $(N,\nIP{\cdot}{\cdot})$ is smooth.  Let $f:\Sigma\times[0,T)\rightarrow N$ be a solution of \eqref{EQwf} and
let $\gamma$ be as in \eqref{EQgamma} satisfying the assumptions \eqref{EQambientcurvass} and \eqref{EQambientinjass} of
Theorem \ref{TMhs} with $u=\gamma$.
There is an $\veZero$ depending only on $\cg$
and $\SK_i$ for $i\in\{0,1,2\}$ such that if
\begin{equation}
\label{EQconcentass}
\varepsilon = \sup_{[0,T)}\int_{[\gamma>0]}|A|^2d\mu
            \le \veZero
\end{equation}
then for any $t\in[0,T)$ we have
\begin{equation}
\label{EQconcentfirstest}
  \begin{split}
  &\int_{[\gamma=1]} |A|^2 d\mu\bigg|_{\tau=t}
 + \int_0^t\int_{[\gamma=1]} \big(|\nabla_{(2)}A|^2 + |\nabla A|^2|A|^2 + |A|^6\big)\,d\mu\,d\tau \\
&\qquad\qquad\qquad
 \le \int_{[\gamma > 0]} |A|^2 d\mu\bigg|_{\tau=0}
 + c_1t\bigg(
              \sup_{[0,T)}\int_{[\gamma>0]}\big(|A|^2 + |\overline{\nabla_{(3)}\Rc}|^2\big)\,d\mu
       \bigg)\,,
  \end{split}
\end{equation}
where $c_1\in(0,\infty)$ is a constant depending only on $\cg$ and $\SK_i$ for $i\in\{0,1,2\}$.
\end{prop}
\begin{proof}
Using Lemma \ref{LMmultsob} we estimate the right hand side of the energy estimate Lemma
\ref{LMenergybasic} as follows.
\begin{align*}
\rD{}{t}&\int_\Sigma |A|^2\gamma^4d\mu
     + \int_\Sigma \big(|\nabla_{(2)}A|^2 + |\nabla A|^2|A|^2 + |A|^6\big)\gamma^4d\mu
     + (1-\delta_0)\int_\Sigma |\nabla_{(2)}A|^2\gamma^4d\mu
\\
 &\le
       c\int_{[\gamma>0]}|A|^2d\mu
   + c\int_\Sigma \big(|\nabla A|^2|A|^2 + |A|^6\big)\gamma^4d\mu
   + c\int_\Sigma|\overline{\nabla_{(3)}\Rc}|^2\,\gamma^4d\mu
\\
 &\le c_0\varepsilon\int_\Sigma \big(|\nabla_{(2)}A|^2 + |A|^6\big)\gamma^4d\mu
     + c(1+\varepsilon)\int_{[\gamma>0]}|A|^2d\mu
   + c\int_\Sigma|\overline{\nabla_{(3)}\Rc}|^2\,\gamma^4d\mu
\,.
\end{align*}
Choosing $\delta_0 = \frac12$, clearly for $\varepsilon$ smaller than $\frac1{2c_0}$ we may absorb the
first term on the left and obtain
\begin{align*}
\rD{}{t}\int_\Sigma |A|^2\gamma^4d\mu
     &+ \int_\Sigma \big(|\nabla_{(2)}A|^2 + |\nabla A|^2|A|^2 + |A|^6\big)\gamma^4d\mu
\\
 &\le
      c\int_{[\gamma>0]}|A|^2d\mu
    + c\int_\Sigma|\overline{\nabla_{(3)}\Rc}|^2\,\gamma^4d\mu
\,.
\end{align*}
Integrating the above differential inequality yields the desired estimate \eqref{EQconcentfirstest}.
\end{proof}

We use Lemma \ref{LMenergygrad1} and Lemma \ref{LMenergygrad2} to improve this to pointwise control.

\begin{prop}
\label{PRbootstrappingfirststep}
Suppose $(N,\nIP{\cdot}{\cdot})$ is smooth, let $f:\Sigma\times[0,T)\rightarrow N$ be a solution to \eqref{EQwf} and let
$\gamma$ be as in \eqref{EQgamma} such that the assumptions \eqref{EQambientcurvass} and \eqref{EQambientinjass} of
Theorem \ref{TMhs} with $u=\gamma$ are satisfied.  There is an $\veOne>0$ and $s_1$ depending only on $\cg$ and $\SK_i$
for $i\in\{0,1,2,3\}$ such that if
\begin{equation*}
\varepsilon = \sup_{[0,T)}\int_{[\gamma>0]}|A|^2d\mu
            \le \veOne \le \veZero
\end{equation*}
then for any $t\in[0,T)$ we have
\begin{align*}
\int_\Sigma&|\nabla A|^2\gamma^{s_1}d\mu\bigg|_{t=\tau}
\\
    &+ \int_0^t\int_\Sigma \big(|\nabla_{(3)}A|^2 + |\nabla_{(2)}A|^2|A|^2 + |\nabla A|^4 + |\nabla A|^2|A|^4 +
                                |A|^6\big)\gamma^{s_1}d\mu d\tau
\notag\\&\quad\le
  c_2 + c_2t\bigg(
              \sup_{[0,T)}\int_{[\gamma>0]}\big(|A|^2 + |\overline{\nabla_{(4)}\Rc}|^2\big)\,d\mu
        \bigg)
\,,
\end{align*}
where $c_2$ is a constant depending only on $\cg$, $\SK_i$ for $i\in\{0,1,2,3\}$, $\vn{A}^2_{2,[\gamma>0]}\big|_{t=0}$,
and $\vn{\nabla A}^2_{2,[\gamma>0]}\big|_{t=0}$.
\end{prop}
\begin{proof}
Throughout the proof we apply several of the previous results. This requires that we choose an $s$ sufficiently large to
allow application of each of them.  Of course, since there are only finitely many choices, $s$ is an absolute constant.
We begin by using Lemma \ref{LMmultsob}, the smallness assumption, and estimate \eqref{LMenergygrad1EQ3} to improve the statement of Lemma \ref{LMenergygrad1}
to
\begin{align}
\rD{}{t}\int_\Sigma|\nabla A|^2\gamma^sd\mu
    &+ \frac32\int_\Sigma |\nabla_{(3)}A|^2 \gamma^sd\mu
\notag\\*&\le
     \int_\Sigma \nabla A*\big(P_3^3 + P_5^1\big)(A)\,\gamma^sd\mu
  + c\varepsilon
    + c\int_{[\gamma>0]}|\overline{\nabla_{(4)}\Rc}|^2\,d\mu
\label{LMbootstrappingEQ0}
\,.
\end{align}
We now need to deal with the $P$-style terms above.
Let us begin by estimating
\begin{align}
\int_\Sigma \nabla A*&\big(P_3^3 + P_5^1\big)(A)\,\gamma^sd\mu
\notag\\
 &= \int_\Sigma \big(\nabla_{(3)}A*\nabla A*A*A
                  + \nabla_{(2)}A*\nabla A*\nabla A*A
\notag\\&\qquad
                  + \nabla A*\nabla A*\nabla A*\nabla A
                  + \nabla A*\nabla A*A*A*A*A\big)\,\gamma^sd\mu
\notag\\&\le
    \delta\int_\Sigma |\nabla_{(3)}A|^2\gamma^sd\mu
  + \int_\Sigma \big(|\nabla_{(2)}A|^2|A|^2 + |\nabla A|^4 + |\nabla A|^2|A|^4\big)\gamma^sd\mu
\label{LMbootstrappingEQ1}
\,.
\end{align}
To control the second term on the right hand side we use new multiplicative Sobolev inequalities 
Modifying the proof of \eqref{LMenergygrad2EQ4} we see that the $|\nabla A|^4$ component can be estimated with
\begin{align}
\int_\Sigma|\nabla A|^4\gamma^sd\mu
 \le c\int_\Sigma |\nabla_{(2)}A|^2|A|^2\gamma^sd\mu
  + c\int_\Sigma |A|^6\,\gamma^{s}d\mu
  + c\vn{A}^2_{2,[\gamma>0]}\,.
\label{LMbootstrappingEQ3}
\end{align}
Since we already have good control over $\vn{A}^6_{6,\gamma^s}$, our control over the right hand side is as good as our
control over $\vn{A\nabla_{(2)}A}^2_{2,\gamma^s}$.
The last component is also controlled by the first.
Recall that we assume \eqref{EQambientcurvass} and \eqref{EQambientinjass} are satisfied for $u=\gamma$, and clearly the
support of $\varphi\gamma$ is contained in the support of $\gamma$ for any function $\varphi:\Sigma\rightarrow\R$.
We may thus apply the Hoffman-Spruck inequality to find
\begin{align}
\int_\Sigma |\nabla A|^2|A|^4\gamma^sd\mu
&\le
  c\Big(
     \int_\Sigma \big(|\nabla_{(2)}A|\,|A|^2 + |\nabla A|^2|A| + |\nabla A|\,|A|^3\big)\,\gamma^{\frac{s}{2}}d\mu
\notag\\&\qquad
   + c\int_\Sigma |\nabla A|\,|A|^2\gamma^{\frac{s}{2}-1}d\mu
   \Big)^2
\notag\\&\le
  c\vn{A}^2_{2,[\gamma>0]}
     \int_\Sigma \big(|\nabla_{(2)}A|^2|A|^2 + |\nabla A|^4 + |\nabla A|^2\,|A|^4\big)\gamma^{s}d\mu
\notag\\&\qquad
   + c\vn{A}^2_{2,[\gamma>0]}
      \int_\Sigma |\nabla A|^2|A|^2\,\gamma^{s-2}d\mu
\notag\\&\le
  c\vn{A}^2_{2,[\gamma>0]}
     \int_\Sigma \big(|\nabla_{(2)}A|^2|A|^2 + |\nabla A|^4 + |\nabla A|^2\,|A|^4\big)\gamma^{s}d\mu
\notag\\&\qquad
   + c\vn{A}^2_{2,[\gamma>0]}\int_\Sigma |A|^4\gamma^{s-4}d\mu
\notag\\&\le
  c\vn{A}^2_{2,[\gamma>0]}
     \int_\Sigma \big(|\nabla_{(2)}A|^2|A|^2 + |\nabla A|^4 + |\nabla A|^2\,|A|^4 + |A|^6\big)\gamma^{s}d\mu
\notag\\&\qquad
   + c\vn{A}^4_{2,[\gamma>0]}
\,.
\label{LMbootstrappingEQ4}
\end{align}
The critical term to control is thus $\vn{A\nabla_{(2)}A}^2_{2,\gamma^s}$.
Using integration by parts we first compute that
\begin{align*}
\int_\Sigma |\nabla_{(2)}A|^2|A|^2\gamma^sd\mu
 &\le
    \int_\Sigma \nabla_{(3)}A*\nabla A*A*A\,\gamma^sd\mu
\\&\quad
  + \int_\Sigma \nabla_{(2)}A*\nabla A*A*A*\nabla\gamma\,\gamma^{s-1}d\mu
\\&\le
    \delta\int_\Sigma |\nabla_{(3)}A|^2\gamma^sd\mu
  + c\int_\Sigma |\nabla A|^2|A|^4\gamma^sd\mu
\\&\quad
  + \int_\Sigma \nabla_{(2)}A*\nabla A*A*A*\nabla\gamma\,\gamma^{s-1}d\mu
\\&\le
    \delta\int_\Sigma |\nabla_{(3)}A|^2\gamma^sd\mu
  + c\int_\Sigma |\nabla A|^2|A|^4\gamma^sd\mu
\\*&\quad
  + \frac12\int_\Sigma |\nabla_{(2)}A|^2|A|^2\gamma^{s}d\mu
  + c\int_\Sigma |\nabla A|^2|A|^2\gamma^{s-2}d\mu\,.
\end{align*}
We continue by absorbing, using \eqref{LMbootstrappingEQ3}, \eqref{LMbootstrappingEQ4} and estimating to obtain
\begin{align*}
\int_\Sigma &|\nabla_{(2)}A|^2|A|^2\gamma^sd\mu
\\
&\le
  c\vn{A}^2_{2,[\gamma>0]}
     \int_\Sigma \big(|\nabla_{(2)}A|^2|A|^2 + |\nabla A|^4 + |\nabla A|^2\,|A|^4 + |A|^6\big)\gamma^{s}d\mu
\\&\qquad
   + \delta\int_\Sigma |\nabla_{(3)}A|^2\gamma^sd\mu
   + c\vn{A}^4_{2,[\gamma>0]}
   + c\int_\Sigma |\nabla A|^2|A|^2\gamma^{s-2}d\mu
\\&\le
  c\vn{A}^2_{2,[\gamma>0]}
     \int_\Sigma \big(|\nabla_{(2)}A|^2|A|^2 + |\nabla A|^4 + |\nabla A|^2\,|A|^4 + |A|^6\big)\gamma^{s}d\mu
\\&\qquad
   + \delta\int_\Sigma |\nabla_{(3)}A|^2\gamma^sd\mu
   + c\vn{A}^4_{2,[\gamma>0]}
\\&\qquad
   + \delta\int_\Sigma |\nabla A|^4\gamma^{s}d\mu
   + c\int_\Sigma |A|^4\gamma^{s-4}d\mu
\\&\le
  c\vn{A}^2_{2,[\gamma>0]}
     \int_\Sigma \big(|\nabla_{(2)}A|^2|A|^2 + |\nabla A|^4 + |\nabla A|^2\,|A|^4 + |A|^6\big)\gamma^{s}d\mu
\\&\qquad
   + \delta\int_\Sigma |\nabla_{(3)}A|^2\gamma^sd\mu
   + c\vn{A}^2_{2,[\gamma>0]}
   + c\vn{A}^4_{2,[\gamma>0]}
\\&\qquad
   + c\tilde\delta\int_\Sigma |\nabla_{(2)}A|^2|A|^2\gamma^sd\mu
   + c\int_\Sigma |A|^6\gamma^{s}d\mu
\,.
\end{align*}
Absorbing once again by choosing $\tilde\delta$ sufficiently small and using Lemma \ref{LMmultsob} followed by \eqref{LMenergygrad1EQ3} we finally have
\begin{align}
\int_\Sigma |\nabla_{(2)}A|^2&|A|^2\gamma^sd\mu
\notag\\
&\le
  c\vn{A}^2_{2,[\gamma>0]}
     \int_\Sigma \big(|\nabla_{(2)}A|^2|A|^2 + |\nabla A|^4 + |\nabla A|^2\,|A|^4 + |A|^6\big)\gamma^{s}d\mu
\notag\\&\qquad
   + \delta\int_\Sigma |\nabla_{(3)}A|^2\gamma^sd\mu
   + c\vn{A}^2_{2,[\gamma>0]}
   + c\vn{A}^4_{2,[\gamma>0]}
\,.
\label{LMbootstrappingEQ5}
\end{align}
Now combining each of the estimates \eqref{LMbootstrappingEQ1}, \eqref{LMbootstrappingEQ3}, \eqref{LMbootstrappingEQ4},
\eqref{LMbootstrappingEQ5} with \eqref{LMbootstrappingEQ0} and choosing $\delta$ sufficiently small we have
\begin{align*}
\rD{}{t}\int_\Sigma|\nabla A|^2\gamma^sd\mu
    &+ \frac54\int_\Sigma \big(|\nabla_{(3)}A|^2 + |\nabla_{(2)}A|^2|A|^2 + |\nabla A|^4 + |\nabla A|^2|A|^4 + |A|^6\big)\gamma^sd\mu
\notag\\&\le
    c\varepsilon\int_\Sigma \big(|\nabla_{(2)}A|^2|A|^2 + |\nabla A|^4 + |\nabla A|^2|A|^4\big)\gamma^sd\mu
\notag\\&\qquad
  + c\varepsilon
  + c\int_{[\gamma>0]}|\overline{\nabla_{(4)}\Rc}|^2\,d\mu
\,.
\end{align*}
Absorbing again for $\veOne$ sufficiently small and integrating finishes the proof.
\end{proof}

\begin{prop}
\label{PRbootstrappingsecondstep}
Suppose $(N,\nIP{\cdot}{\cdot})$ is smooth, let $f:\Sigma\times[0,T)\rightarrow N$ be a solution to \eqref{EQwf} and let
$\gamma$ be as in \eqref{EQgamma} such that the assumptions \eqref{EQambientcurvass} and \eqref{EQambientinjass} of
Theorem \ref{TMhs} with $u=\gamma$ are satisfied.
There is an $\veTwo>0$ and $s_2$ depending only on $\cg$ and $\SK_i$ for $i\in\{0,1,2,3,4\}$ such that if
\begin{equation}
\varepsilon = \sup_{[0,T)}\int_{[\gamma>0]}|A|^2d\mu
            \le \veTwo \le \veOne \le \veZero
\label{EQeps2small}
\end{equation}
then for any $t\in[0,T)$ we have
\begin{align*}
\int_\Sigma|\nabla_{(2)} A|^2\gamma^{s_2}d\mu\bigg|_{\tau=t}
    &+ \frac54\int_0^t
       \int_\Sigma \big(
                     |\nabla_{(2)}A|^2|A|^2
                   + |\nabla A|^4
                   + |\nabla A|^2\,|A|^4
\notag\\&\qquad\qquad
                   + |A|^6
                   + |\nabla_{(2)}A|^2|A|^4 + |\nabla_{(3)}A|^2
                   + |\nabla_{(3)}A|^2|A|^2
\notag\\&\qquad\qquad
                   + |\nabla_{(2)} A|^2|\nabla A|^2
                   + |\nabla A|^2|A|^6 \big)\gamma^{s_2}d\mu\,d\tau
\notag\\&\le
    c_3 + c_3t\bigg(
              \sup_{[0,T)}\int_{[\gamma>0]}\big(|A|^2 + |\overline{\nabla_{(5)}\Rc}|^2\big)\,d\mu
        \bigg)
\,.
\end{align*}
where $c_3$ is a constant depending only on $\cg$, $\SK_i$ for $i\in\{0,1,2,3,4\}$, $T$, and $\vn{\nabla_{(i)}
A}^2_{2,[\gamma>0]}\big|_{t=0}$ for $i\in\{0, 1, 2\}$.
\end{prop}
\begin{proof}
We begin by using Lemma \ref{LMmultsob}, the smallness assumption, and estimates \eqref{LMenergygrad1EQ11}, \eqref{LMenergygrad2EQ2} to improve the statement of
Lemma \ref{LMenergygrad2} to
\begin{align}
\rD{}{t}\int_\Sigma|\nabla_{(2)} A|^2\gamma^sd\mu
    &+ \frac32\int_\Sigma |\nabla_{(4)}A|^2 \gamma^sd\mu
\notag\\&\le
     \int_\Sigma \nabla_{(2)} A*\big(P_3^4 + P_5^2\big)(A)\,\gamma^sd\mu
  + c\int_{[\gamma>0]} \big(|A|^2 + |\overline{\nabla_{(5)}\Rc}|^2\big)d\mu
\label{LMbootstrapping2EQ0}
\,.
\end{align}
We need to deal with the $P$-style terms above.
Let us use integration by parts to estimate
\begin{align}
\int_\Sigma& \nabla_{(2)} A*\big(P_3^4 + P_5^2\big)(A)\,\gamma^sd\mu
\notag\\
 &= \int_\Sigma \big(\nabla_{(2)}A*\nabla_{(4)} A*A*A
                  + \nabla_{(2)}A*\nabla_{(3)}A*\nabla A*A
\notag\\&\quad
                  + \nabla_{(2)}A*\nabla_{(2)}A*\nabla_{(2)} A*A
                  + \nabla_{(2)}A*\nabla_{(2)}A*\nabla A*\nabla A
\notag\\&\quad
                  + \nabla_{(2)}A*\nabla_{(2)} A*A*A*A*A
                  + \nabla_{(2)}A*\nabla A*\nabla A*A*A*A\big)\,\gamma^sd\mu
\notag\\&\le
    \delta\int_\Sigma |\nabla_{(4)}A|^2\gamma^sd\mu
  + c\int_\Sigma \big(|\nabla_{(2)}A|^2|A|^4 + |\nabla_{(3)}A|^2|A|^2
\notag\\&\qquad\qquad\qquad\qquad\qquad
                   + |\nabla_{(2)} A|^2|\nabla A|^2 + |\nabla_{(2)}A|^2 + |\nabla A|^2|A|^6
                \big)\gamma^sd\mu
\label{LMbootstrapping2EQ1}
\,.
\end{align}
In order to deal with these terms we use the following interpolation inequality. The proof is by induction on an
easy modification of the argument used to
prove \eqref{LMenergygrad1EQ3}.
For each $\delta>0$ there exist constants $c\in(0,\infty)$ and $s\in(4,\infty)$ depending only on $\delta$, $n$, $p$,
$r$, $\cg$ and $k$ such that 
\begin{equation}
\int_\Sigma |\nabla_{(k)}A|^2\gamma^{s-p}d\mu
 \le \delta\int_\Sigma |\nabla_{(k+q)}A|^2\gamma^{s+r}d\mu + c\int_{[\gamma>0]} |A|^2d\mu\,,
\label{LMbootstrapping2EQ2}
\end{equation}
We will also need the estimate
\begin{equation}
\int_\Sigma |A|^6 \gamma^{s-p}d\mu \le c\varepsilon
\label{LMbootstrapping2EQ3}
\end{equation}
which holds for $s \ge \max\{s_1+p,\, 2p+4\}$, and for $c$ depending additionally on $T$. This dependence is affine, so that $T<\infty$ gives $c<\infty$.
The proof of this is similar to that of Lemma \ref{LMmultsob}, except we use the additional information given to us by Proposition
\ref{PRbootstrappingfirststep}. For the reader's convenience we reproduce the details here. Applying Theorem \ref{TMhs} we have
\begin{align}
\int_\Sigma |A|^6\gamma^sd\mu
 &\le c\bigg(
              \int_\Sigma |\nabla A|\,|A|^2\gamma^{\frac{s}{2}}d\mu
            + \int_\Sigma |A|^3\gamma^{\frac{s}{2}}d\mu
            + \int_\Sigma |A|^4\gamma^{\frac{s}{2}}d\mu
       \bigg)^2
\notag\\
 &\le c\bigg(
              \int_\Sigma |\nabla A|\,|A|^2\gamma^{\frac{s}{2}}d\mu
            + \int_\Sigma |A|^2\gamma^{\frac{s}{2}}d\mu
            + c\int_\Sigma |A|^4\gamma^{\frac{s}{2}}d\mu
       \bigg)^2
\notag\\
 &\le c\int_\Sigma |\nabla A|^2\gamma^{\frac{s}{2}}d\mu
       \int_\Sigma |A|^4\gamma^{\frac{s}{2}}d\mu
    + c\varepsilon\int_\Sigma |A|^6\gamma^{s}d\mu
    + c\varepsilon
\notag\\
 &\le
      c\Big(\int_\Sigma |\nabla A|^2\gamma^{\frac{s}{2}}d\mu\Big)^2
       \int_{[\gamma>0]} |A|^2d\mu
    + \frac12\int_\Sigma |A|^6\gamma^{s}d\mu
\notag\\&\qquad
    + c\varepsilon\int_\Sigma |A|^6\gamma^{s}d\mu
    + c\varepsilon
\notag\\
 &\le
      \frac12\int_\Sigma |A|^6\gamma^{s}d\mu
    + c\varepsilon\int_\Sigma |A|^6\gamma^{s}d\mu
    + c\varepsilon\,,
\label{LMbootstrapping2EQreadconvenience}
\end{align}
whereupon absorbing finishes the proof of \eqref{LMbootstrapping2EQ3}.
We applied Proposition \ref{PRbootstrappingfirststep} in the last estimate, which requires $\frac{s}{2}\ge s_1$.
Note that the constant depends on $T$ and the initial data, but is finite so long as $T<\infty$ and $f_0 \in W^{3,2}$.

We now continue with the main proof.
Applying Theorem \ref{TMhs} we estimate the last term on the right hand side of \eqref{LMbootstrapping2EQ1} by
\begin{align}
\int_\Sigma |\nabla A|^2|A|^6\gamma^sd\mu
 &\le c\bigg(
              \int_\Sigma |\nabla_{(2)} A|\,|A|^3\gamma^{\frac{s}{2}}d\mu
            + \int_\Sigma |\nabla A|^2\,|A|^2\gamma^{\frac{s}{2}}d\mu
\notag\\&\quad
            + \int_\Sigma |\nabla A|\,|A|^3\gamma^{\frac{s}{2}-1}d\mu
            + \int_\Sigma |\nabla A|\,|A|^4\gamma^{\frac{s}{2}}d\mu
       \bigg)^2
\notag\\
 &\le c\varepsilon\int_\Sigma |\nabla_{(2)} A|^2|A|^2\gamma^{s}d\mu
    + c\varepsilon\int_\Sigma |\nabla A|^2\,|A|^6\gamma^{s}d\mu
\notag\\&\quad
            + c\bigg(\int_\Sigma |\nabla A|^2\,|A|^2\gamma^{\frac{s}{2}}d\mu\bigg)^2
            + c\int_\Sigma |A|^6\gamma^{\frac{s}{2}-2}d\mu
              \int_\Sigma |\nabla A|^2\gamma^{\frac{s}{2}}d\mu
\notag\\
 &\le c\varepsilon\int_\Sigma |\nabla_{(2)} A|^2|A|^2\gamma^{s}d\mu
    + c\varepsilon\int_\Sigma |\nabla A|^2\,|A|^6\gamma^{s}d\mu
\notag\\&\quad
            + c\bigg(\int_\Sigma |\nabla A|^2\,|A|^2\gamma^{\frac{s}{2}}d\mu\bigg)^2
            + c\varepsilon
              \int_\Sigma |\nabla A|^2\gamma^{\frac{s}{2}}d\mu\,,
\notag\\
 &\le c\varepsilon\int_\Sigma |\nabla_{(2)} A|^2|A|^2\gamma^{s}d\mu
    + c\varepsilon\int_\Sigma |\nabla A|^2\,|A|^6\gamma^{s}d\mu
\notag\\&\quad
            + c\bigg(\int_\Sigma |\nabla A|^2\,|A|^2\gamma^{\frac{s}{2}}d\mu\bigg)^2
            + c\varepsilon
\label{LMbootstrapping2EQ4}
\end{align}
using \eqref{LMbootstrapping2EQ3} in the last two steps and requiring $s \ge \max\{2s_1+2p,\, 4p+8\}$.
We continue by using Theorem \ref{TMhs}, integration by parts, and standard estimates to control the second last term:
\begin{align*}
\int_\Sigma &|\nabla A|^2|A|^2\gamma^{\frac{s}{2}}d\mu
\\*&\le c\bigg(
           \int_\Sigma \big(|\nabla_{(2)}A|\cdot|A| + |\nabla A|^2 + |\nabla A|\cdot|A|^2\big)\gamma^{\frac{s}{4}} + |\nabla
A|\cdot|A|\gamma^{\frac{s}{4}-1}\,d\mu
        \bigg)^2
\\
  &\le
       c\varepsilon
        \int_\Sigma |\nabla A|^2|A|^2\gamma^{\frac{s}{2}}d\mu
     + c\varepsilon
        \int_\Sigma |\nabla_{(2)}A|\gamma^{\frac{s}{2}}d\mu
     + c\varepsilon
        \int_\Sigma |\nabla A|^2\gamma^{\frac{s}{2}-2}d\mu\,,
\end{align*}
so that \eqref{LMbootstrapping2EQ3} (with $s \ge 2s_1+4$), Proposition \ref{PRbootstrappingfirststep}, and absorbing give
\begin{align}
\int_\Sigma |\nabla A|^2|A|^2\gamma^{\frac{s}{2}}d\mu
  &\le
       c\varepsilon
        \int_\Sigma |\nabla_{(2)}A|^2\gamma^{\frac{s}{2}}d\mu
     + c\varepsilon\,.
\label{LMbootstrapping2EQ5}
\end{align}
Squaring \eqref{LMbootstrapping2EQ5} and estimating with \eqref{LMbootstrapping2EQ2} we arrive at
\begin{align}
\bigg(\int_\Sigma |\nabla A|^2|A|^2\gamma^{\frac{s}{2}}d\mu\bigg)^2
  &\le
       c\varepsilon
        \int_\Sigma |\nabla_{(3)}A|^2\gamma^{{s}}d\mu
     + c\varepsilon\,.
\label{LMbootstrapping2EQ6}
\end{align}
Combining \eqref{LMbootstrapping2EQ6} with \eqref{LMbootstrapping2EQ4} and absorbing we have
\begin{align}
\int_\Sigma |\nabla A|^2|A|^6\gamma^sd\mu
 &\le c\varepsilon\int_\Sigma |\nabla_{(2)} A|^2|A|^2\gamma^{s}d\mu
     + c\varepsilon
        \int_\Sigma |\nabla_{(3)}A|^2\gamma^{{s}}d\mu
     + c\varepsilon
\,.
\label{LMbootstrapping2EQ7}
\end{align}
The first term on the right hand side of \eqref{LMbootstrapping2EQ7} can be estimated exactly as in the proof of
Proposition \ref{PRbootstrappingfirststep} and the second term can be estimated via \eqref{LMbootstrapping2EQ2}.
In light of the work completed in the proof of Proposition \ref{PRbootstrappingfirststep} we have thus improved
\eqref{LMbootstrapping2EQ1} to
\begin{align}
\int_\Sigma \nabla_{(2)} A*&\big(P_3^4 + P_5^2\big)(A)\,\gamma^sd\mu
   + \int_\Sigma \big(|\nabla_{(3)}A|^2 + |\nabla A|^2|A|^6 \big)\gamma^{s}d\mu
\notag\\&\quad
   + \int_\Sigma \big(|\nabla_{(2)}A|^2|A|^2 + |\nabla A|^4 + |\nabla A|^2\,|A|^4 + |A|^6\big)\gamma^{s}d\mu
\notag\\&\le
    \delta\int_\Sigma |\nabla_{(4)}A|^2\gamma^sd\mu
\notag\\&\qquad
  + c\varepsilon
     \int_\Sigma \big(|\nabla_{(2)}A|^2|A|^2 + |\nabla A|^4 + |\nabla A|^2\,|A|^4 + |A|^6\big)\gamma^{s}d\mu
\notag\\&\qquad
  + c\int_\Sigma \big(|\nabla_{(2)}A|^2|A|^4 + |\nabla_{(3)}A|^2|A|^2
                   + |\nabla_{(2)} A|^2|\nabla A|^2
                \big)\gamma^sd\mu
   + c\varepsilon
\label{LMbootstrapping2EQ8}
\,.
\end{align}
The second integral on the right hand side of \eqref{LMbootstrapping2EQ8} can be absorbed for $\veZero$ sufficiently small. The first integrand of the
third integral is controlled by the second two, thanks to the following estimate.
Using the Hoffman-Spruck Sobolev inequality and \eqref{LMbootstrappingEQ5}, we have
\begin{align*}
\int_\Sigma &|\nabla_{(2)}A|^2|A|^4\gamma^sd\mu
\\
 &\le c\bigg(
         \int_\Sigma \big(
                       |\nabla_{(3)}A|\cdot|A|^2 + |\nabla_{(2)}A|\cdot|\nabla A|\cdot|A| + |\nabla_{(2)}A|\cdot|A|^3
                     \big)\gamma^{\frac{s}{2}}
\\&\qquad
                   + |\nabla_{(2)}A|\cdot|A|^2\gamma^{\frac{s}{2}-1}d\mu
       \bigg)^2
\\&\le
     c\varepsilon\int_\Sigma |\nabla_{(3)}A|^2|A|^2\gamma^s d\mu
   + c\varepsilon\int_\Sigma |\nabla_{(2)}A|^2|\nabla A|^2\gamma^s d\mu
\\*&\qquad
   + c\varepsilon\int_\Sigma |\nabla_{(3)}A|^2\gamma^s d\mu
   + c\varepsilon\int_\Sigma |\nabla_{(2)}A|^2|A|^4\gamma^sd\mu
\\*&\qquad
   + c\varepsilon\,,
\end{align*}
which upon absorption yields
\begin{align}
\int_\Sigma |\nabla_{(2)}A|^2|A|^4\gamma^sd\mu
 &\le 
     c\varepsilon\int_\Sigma |\nabla_{(3)}A|^2|A|^2\gamma^s d\mu
   + c\varepsilon\int_\Sigma |\nabla_{(2)}A|^2|\nabla A|^2\gamma^s d\mu
\notag\\&\qquad
   + c\varepsilon\int_\Sigma |\nabla_{(3)}A|^2\gamma^s d\mu
   + c\varepsilon\,.
\label{LMbootstrapping2EQ9}
\end{align}
Combining \eqref{LMbootstrapping2EQ9} with \eqref{LMbootstrapping2EQ8} and absorbing (for $\veZero$ sufficiently small) we find
\begin{align}
\int_\Sigma \nabla_{(2)} A*&\big(P_3^4 + P_5^2\big)(A)\,\gamma^sd\mu
   + \int_\Sigma \big(|\nabla_{(2)}A|^2|A|^4 + |\nabla_{(3)}A|^2 + |\nabla A|^2|A|^6 \big)\gamma^{s}d\mu
\notag\\&\quad
   + \int_\Sigma \big(|\nabla_{(2)}A|^2|A|^2 + |\nabla A|^4 + |\nabla A|^2\,|A|^4 + |A|^6\big)\gamma^{s}d\mu
\notag\\&\le
    \delta\int_\Sigma |\nabla_{(4)}A|^2\gamma^sd\mu
  + c\int_\Sigma \big(
                     |\nabla_{(3)}A|^2|A|^2
                   + |\nabla_{(2)} A|^2|\nabla A|^2
                \big)\gamma^sd\mu
   + c\varepsilon
\label{LMbootstrapping2EQ10}
\,.
\end{align}
The second integrand in the second integral on the right hand side of \eqref{LMbootstrapping2EQ10} can be estimated by
\begin{align*}
\int_\Sigma |\nabla_{(2)}A|^2&|\nabla A|^2\gamma^sd\mu
\\
 &\le c\bigg(
         \int_\Sigma \big(
                       |\nabla_{(3)}A|\cdot|\nabla A| + |\nabla_{(2)}A|^2 + |\nabla_{(2)}A|\cdot|\nabla A|\cdot|A|
                     \big)\gamma^{\frac{s}{2}}
\\&\qquad
                     + |\nabla_{(2)}A|\cdot|\nabla A|\gamma^{\frac{s}{2}-1}d\mu
       \bigg)^2
\\&\le
     c\bigg(\int_\Sigma |\nabla_{(3)}A|\cdot|\nabla A|\gamma^{\frac{s}{2}} d\mu\bigg)^2
   + c\bigg(\int_\Sigma |\nabla_{(2)}A|\cdot|\nabla A|\gamma^{\frac{s}{2}-1} d\mu\bigg)^2
\\&\qquad
   + c\bigg(\int_\Sigma |\nabla_{(2)}A|^2\gamma^{\frac{s}{2}} d\mu\bigg)^2
   + c\varepsilon\int_\Sigma |\nabla_{(2)}A|^2|\nabla A|^2\gamma^sd\mu
\\&\le
     c\bigg(\int_\Sigma |\nabla_{(3)}A|\cdot|\nabla A|\gamma^{\frac{s}{2}} d\mu\bigg)^2
   + c\bigg(\int_\Sigma |\nabla_{(2)}A|\cdot|\nabla A|\gamma^{\frac{s}{2}-1} d\mu\bigg)^2
\\&\qquad
   + c\varepsilon\int_\Sigma |\nabla_{(2)}A|^2|\nabla A|^2\gamma^sd\mu\,,
\end{align*}
which upon absorption and using Proposition \ref{PRbootstrappingfirststep} yields
\begin{align}
\int_\Sigma |\nabla_{(2)}A|^2&|\nabla A|^2\gamma^sd\mu
\notag\\
&\le
     c\bigg(\int_\Sigma |\nabla_{(3)}A|\cdot|\nabla A|\gamma^{\frac{s}{2}} d\mu\bigg)^2
   + c\bigg(\int_\Sigma |\nabla_{(2)}A|\cdot|\nabla A|\gamma^{\frac{s}{2}-1} d\mu\bigg)^2
\notag\\
&\le
     c\int_\Sigma |\nabla_{(3)}A|^2\gamma^{s-s_1}d\mu
   + c\int_\Sigma |\nabla_{(2)}A|^2\gamma^{s-s_1-2}d\mu\,.
\label{LMbootstrapping2EQ11}
\end{align}
Note that $s$ is already assumed larger than required for the last step. To estimate this pair of integrals we apply \eqref{LMbootstrapping2EQ2} twice to obtain
\begin{align}
\int_\Sigma |\nabla_{(2)}A|^2|\nabla A|^2\gamma^sd\mu
&\le
     \delta\int_\Sigma |\nabla_{(4)}A|^2\gamma^{s}d\mu
   + c\varepsilon\,.
\label{LMbootstrapping2EQ12}
\end{align}
Plugging the above estimate \eqref{LMbootstrapping2EQ12} into \eqref{LMbootstrapping2EQ10} we find
\begin{align}
\int_\Sigma \nabla_{(2)} A*&\big(P_3^4 + P_5^2\big)(A)\,\gamma^sd\mu
   + \int_\Sigma \big(|\nabla_{(2)}A|^2|A|^2 + |\nabla A|^4 + |\nabla A|^2\,|A|^4 + |A|^6\big)\gamma^{s}d\mu
\notag\\&\quad
   + \int_\Sigma \big(|\nabla_{(2)}A|^2|A|^4 + |\nabla_{(3)}A|^2
                   + |\nabla_{(2)} A|^2|\nabla A|^2
                   + |\nabla A|^2|A|^6 \big)\gamma^{s}d\mu
\notag\\&\le
    \delta\int_\Sigma |\nabla_{(4)}A|^2\gamma^sd\mu
  + c\int_\Sigma \big(
                     |\nabla_{(3)}A|^2|A|^2
                \big)\gamma^sd\mu
   + c\varepsilon
\label{LMbootstrapping2EQ13}
\,.
\end{align}
Estimating the remaining term on the right hand side of \eqref{LMbootstrapping2EQ13} is similar to the estimates we have
already obtained:
\begin{align}
\int_\Sigma |\nabla_{(3)}A|^2|A|^2\gamma^sd\mu
 &\le c\bigg(
         \int_\Sigma \big(
                       |\nabla_{(4)}A|\cdot|A| + |\nabla_{(3)}A|\cdot|\nabla A| + |\nabla_{(3)}A|\cdot|A|^2
                     \big)\gamma^{\frac{s}{2}}
\notag\\&\qquad
                     + |\nabla_{(3)}A|\cdot|A|\gamma^{\frac{s}{2}-1}d\mu
       \bigg)^2
\notag\\
&\le
     c\varepsilon\int_\Sigma |\nabla_{(4)}A|^2\gamma^sd\mu
   + c\varepsilon\int_\Sigma |\nabla_{(3)}A|^2|A|^2\gamma^sd\mu
\notag\\
&\qquad
   + c\varepsilon\int_\Sigma |\nabla_{(3)}A|^2\gamma^{s-s_1-2}d\mu\,.
\notag
\end{align}
Using again the estimate \eqref{LMbootstrapping2EQ2} and absorbing we finally have
\begin{align}
\int_\Sigma |\nabla_{(3)}A|^2|A|^2\gamma^sd\mu
&\le
     c\varepsilon\int_\Sigma |\nabla_{(4)}A|^2\gamma^sd\mu
   + c\varepsilon\,,
\notag
\end{align}
which we combine with \eqref{LMbootstrapping2EQ13} to obtain
\begin{align}
\int_\Sigma \nabla_{(2)} A*&\big(P_3^4 + P_5^2\big)(A)\,\gamma^sd\mu
\notag\\&
   + \int_\Sigma \big(|\nabla_{(2)}A|^2|A|^2 + |\nabla A|^4 + |\nabla A|^2\,|A|^4 + |A|^6\big)\gamma^{s}d\mu
\notag\\&
   + \int_\Sigma \big(|\nabla_{(2)}A|^2|A|^4 + |\nabla_{(3)}A|^2
                   + |\nabla_{(3)}A|^2|A|^2
\notag\\&\qquad\qquad\qquad\qquad
                   + |\nabla_{(2)} A|^2|\nabla A|^2
                   + |\nabla A|^2|A|^6 \big)\gamma^{s}d\mu
\notag\\&\le
    \delta\int_\Sigma |\nabla_{(4)}A|^2\gamma^sd\mu
   + c\varepsilon
\label{LMbootstrapping2EQ14}
\,.
\end{align}
Combining the estimate \eqref{LMbootstrapping2EQ14} with \eqref{LMbootstrapping2EQ0}, choosing $\delta$, $\veTwo$ sufficiently small and absorbing yields
the result.
\end{proof}

These last two estimates combine to give a pointwise curvature bound which is crucial.

\begin{cor}
\label{CYpointwisebound1}
Suppose $(N,\nIP{\cdot}{\cdot})$ is smooth, simply connected, complete with non-positive sectional curvature.
Suppose $f:\Sigma\times[0,T)\rightarrow N$ is a solution to \eqref{EQwf} satisfying for some $\rho>0$
\begin{equation*}
\sup_{[0,T)}\sup_{x\in N} \int_{f^{-1}(B_\rho(x))} |A|^2d\mu
            \le \veTwo\,
\end{equation*}
and
\[
\sup_{[0,T)}\sup_{x\in N}\int_{f^{-1}(B_\rho(x))} |\overline{\nabla_{(i)}\Rc}|^2\,d\mu \le \cK_{\rho,i} < \infty\,.
\]
Then for any $t\in[0,T)$ we have
\begin{equation*}
\vn{A}_{\infty} \le c\,\varepsilon\,,
\end{equation*}
where $c$ is a constant depending only on $\cg$, $\SK_i$ for $i\in\{0,1,2,3,4\}$, $\cK_{\rho,i}$ for $i\in\{3,4,5\}$,
and $\vn{\nabla_{(i)} A}^2_{2,[\gamma>0]}$ for $i\in\{1, 2\}$.
\end{cor}
\begin{proof}
Let $x\in N$ be arbitrary and set $\gamma$ to be as in \eqref{EQgamma}, satisfying in addition
\[
\chi_{B_{\frac\rho2}(x)} \le \tilde\gamma \le \chi_{B_{\rho}(x)}\,.
\]
Lemma 4.3 from \cite{kuwert2002gfw} holds whenever Theorem \ref{TMhs} holds, which is for any $u$, due to the hypotheses
on the ambient space.
Thus, applying Proposition \ref{PRbootstrappingsecondstep},
\[
\vn{A}_{\infty} =
\vn{A}_{\infty,f^{-1}(B_\rho(x))} =
\vn{A}_{\infty,[\gamma=1]}
 \le c\varepsilon\big(
                \vn{\nabla_{(2)}A}_{2,[\gamma>0]}^2 + \varepsilon
                 \big)
 \le c\varepsilon\,.
\]
\end{proof}

For solutions of \eqref{EQwf} in 3-manifolds which are not simply-connected or have somewhere positive Ricci curvature,
it is not clear that the Hoffman-Spruck Sobolev inequality is applicable along the flow.  The surface area may grow to
violate one or both of \eqref{EQambientcurvass}, \eqref{EQambientinjass}.

This is somewhat tricky: it is absolutely critical that the Sobolev constant in Theorem \ref{TMhs} does \emph{not}
depend on the geometry of $f(\Sigma)$.  Our strategy is to show that if the initial \emph{concentration of area} is
small enough, \emph{and} we have good control on certain curvature integrals in $L^1([0,T))$, then the concentration of
area for a distinct time interval remains small.

\begin{prop}[Control on the concentration of area]
\label{PRconcentarea}
Let $f:\Sigma\times[0,T)\rightarrow N$ be a solution to \eqref{EQwf}.
Set
\[
\csupp := 
  \min\Bigg\{\frac{8\pi}{9\SK}\,,\,
            \frac{4\pi\sin^2\big(2\rho_N\sqrt\SK\big)}{9\SK}
      \bigg\}\,.
\]
There is a tuple $(\hat\varepsilon_0, \hosigma)$ depending only on the metric of $N$ and $\cg$ such that the following
statement holds.
Assume $\rho>0$ is such that
\begin{equation}
\sup_{[0,T)}\sup_{x\in N} \int_{f^{-1}(B_\rho(x))} |A|^2d\mu
            \le \hat\varepsilon_0 \le \min\{\veZero,\veOne,\veTwo\}
\label{EQconcentsmallsupass}
\end{equation}
where $\veZero$, $\veOne$, and $\veTwo$ are from Propositions \ref{PRconcentestvanilla}, \ref{PRbootstrappingfirststep},
\ref{PRbootstrappingsecondstep}, and
\begin{equation}
\sigma := \sup_{x\in N}|\Sigma|_{f^{-1}(B_\rho(x))}\big|_{t=0}
\le \hsigma_0 
< \csupp\,.
\label{EQinitsmallconcareaass}
\end{equation}
Let $\gamma$ be as in \eqref{EQgamma} with $[\gamma>0]\subset f^{-1}(B_\rho(x))$ for some $x\in N$.
Then the estimate
\begin{equation}
\gconc < \csupp
\label{EQsupportestimate}
\end{equation}
is valid for all $t\le \min\{T,t_0\} =: T^*(\hsigma_0,\hvarepsilon,\SK_0)$, where
\[
t_0 \ge \sqrt[4]{\frac{\csupp}{c_4(\hsigma_0+\hvarepsilon)}}-1\,,
\]
and $c_4$ is a constant depending only on the metric of $N$, $\cg$, $\SK_i$ for $i\in\{0,1,2,3,4,5\}$, and
$\vn{\nabla_{(i)} A}^2_{2}\big|_{t=0}$ for $i\in\{0,1,2\}$.  In particular, the conditions \eqref{EQambientcurvass} and
\eqref{EQambientinjass} are satisfied for any $u$ with compact support $[u>0]\subset[\gamma=1]$ on the time interval
$[0,T^*)$.
\end{prop}
\begin{proof}
Using Lemma \ref{LMevolutionequations} we compute
\begin{align*}
\rD{}{t}\int_\Sigma \gamma^4 d\mu
 &= -4\int_\Sigma (D_\nu\tilde{\gamma})(\Delta H + H|A^o|^2 + H\Rcn(\nu,\nu))\,\gamma^3d\mu
\\&\quad
  + \int_\Sigma H(\Delta H + H|A^o|^2 + H\Rcn(\nu,\nu))\,\gamma^4d\mu
\\
 &= -4\int_\Sigma \sIP{\nabla D_\nu\tilde{\gamma}}{\nabla H}\,\gamma^3d\mu
   -12\int_\Sigma \sIP{D_\nu\tilde{\gamma}\nabla\gamma}{\nabla H}\,\gamma^2d\mu
\\&\quad
  -4\int_\Sigma (D_\nu\tilde{\gamma})(H|A^o|^2 + H\Rcn(\nu,\nu))\,\gamma^3d\mu
\\&\quad
  + \int_\Sigma H(\Delta H + H|A^o|^2 + H\Rcn(\nu,\nu))\,\gamma^4d\mu\,.
\intertext{Estimating, we have}
\rD{}{t}\gconc
 &\le
    c\int_\Sigma (1+|A|)|\nabla A|\,\gamma^2d\mu
  + c\int_\Sigma |A|^4 + |A|^3 + \SK_0|A|^2 + \SK_0|A|\,\gamma^3d\mu
\\&\quad
  + c\int_\Sigma |A|\,|\nabla_{(2)}A|\,\gamma^4d\mu\,.
\\
 &\le
    c\SK_0\sqrt{\gconc}\sqrt{\int_{[\gamma>0]}|A|^2d\mu}
\\&\quad
  + c(1+\SK_0)\int_{[\gamma>0]}|A|^2d\mu
  + c\int_\Sigma |A|^6\gamma^4 d\mu
\\&\quad
  + c\bigg(\int_{[\gamma>0]} |A|^2d\mu\bigg)^{\frac12}\bigg(\int_\Sigma |\nabla_{(2)}A|^2\,\gamma^4 d\mu\bigg)^{\frac12}
\\&\quad
  + c\bigg(\int_{[\gamma>0]} |A|^2d\mu\bigg)^{\frac12}\bigg(\int_\Sigma |\nabla A|^2\,\gamma^4 d\mu\bigg)^{\frac12}
\\&\quad
  + c\sqrt{\gconc}\int_{[\gamma>0]}|\nabla A|^2d\mu\,.
\end{align*}
Since $\gamma$ has compact support, we can cover its image under $f$ with finitely many balls of radius $\rho$, with $\rho$ as in \eqref{EQconcentsmallsupass}.
Therefore
\[
\int_{[\gamma>0]}|A|^2d\mu \le c\hat\varepsilon_0\,,
\]
where $c$ now and for the rest of the proof depends additionally on the metric of $N$.
Smoothness of the solution and the hypothesis \eqref{EQinitsmallconcareaass} implies that
\eqref{EQambientcurvass} and \eqref{EQambientinjass} are satisfied on a maximal time interval $[0,t_0)$, where
$t_0\le T$. Suppose that
\begin{equation}
t_0 < \sqrt{\frac{
 \csupp
}{2c_4(\hsigma_0+\hat\varepsilon_0)}}-1\,,
\label{EQt0}
\end{equation}
for a parameter $c_4$ to be later specified.
We will show that this leads to a contradiction.

Note that for $\hsigma_0$ and $\hat\varepsilon_0$ small the right hand side of \eqref{EQt0} is large.
The conditions \eqref{EQambientcurvass} and \eqref{EQambientinjass} for the Hoffman-Spruck Sobolev inequality are satisfied on $[0,t_0)$.
We may therefore apply Propositions \ref{PRconcentestvanilla}, \ref{PRbootstrappingfirststep} and \ref{PRbootstrappingsecondstep}.
Using these with a cutoff function $\eta$ in these propositions satisfying $[\eta>0] \subset [\gamma=1]$, we have
\[
\int_{[\eta>0]}d\mu \le \int_{[\gamma=1]}d\mu \le \gconc\,.
\]
Presuming $\hvarepsilon \le 1$, this yields
\begin{align}
\rD{}{t}\gconc
 &\le
    c\SK_0\sqrt{\hat\varepsilon_0}\sqrt{\gconc} + c(1+\SK_0)\sqrt{\hvarepsilon}
\notag\\*&\qquad
  + ct_0\hvarepsilon
  + c\sqrt{t_0}\sqrt{\hvarepsilon}(\SK_4 + \SK_5)\sqrt{\gconc}
\notag\\*&\qquad
  + c\sqrt{\gconc}\int_{[\gamma>0]}|\nabla A|^2d\mu\,,
\label{EQconcentareaEQ0}
\end{align}
for $t\in[0,t_0)$.
Now consider a cutoff function $\varphi$ satisfying the conditions of \eqref{EQgamma} and $[\gamma>0] \subset [\varphi = 1]$, so that
\begin{align*}
  \int_{[\gamma>0]}|\nabla A|^2d\mu
  &\le
  \int_{[\varphi=1]}|\nabla A|^2d\mu
   \le
     \int_\Sigma |\nabla A|^2 \varphi^s d\mu
\\&\le
     \int_\Sigma |\nabla_{(2)}A|\,|A| \varphi^s d\mu
   + c\int_\Sigma |\nabla A|\,|A| \varphi^{s-1} d\mu
\\&\le
     \bigg(\int_\Sigma |A|^2\varphi^sd\mu\bigg)^{\frac12}\bigg(\int_\Sigma |\nabla_{(2)}A|^2\,\varphi^s
d\mu\bigg)^{\frac12}
   + \frac12\int_\Sigma |\nabla A|^2\varphi^{s} d\mu
\\&\qquad
   + c\int_\Sigma |A|^2 \varphi^{s-2} d\mu\,.
\end{align*}
Covering $[\varphi > 0]$ with $B_\rho(x_i)$ for a finite set of $x_i\in N$ and using \eqref{EQconcentsmallsupass}, we have that $\vn{A}^2_{2,[\varphi>0]} \le
\veTwo$ and can thus apply Proposition \ref{PRbootstrappingsecondstep} (with $s\ge s_2$) to obtain
\begin{align}
  \int_{[\gamma>0]}|\nabla A|^2d\mu
  &\le
     c\bigg(\int_{[\varphi>0]} |A|^2d\mu\bigg)^{\frac12}\Big(1 + t_0\hvarepsilon + t_0\SK_5^2\gconc\Big)^\frac12
   + c\int_{[\varphi>0]} |A|^2 d\mu
\notag\\
  &\le
     c\sqrt{\hvarepsilon}\Big(1 + \sqrt{t_0\hvarepsilon} + \SK_5\sqrt{t_0}\sqrt{\gconc}\Big)
   + c\hvarepsilon
\,.
\label{EQconcentareaEQ1}
\end{align}
Combining \eqref{EQconcentareaEQ1} with \eqref{EQconcentareaEQ0} yields
\begin{align}
\rD{}{t}\gconc
 &\le
    c\SK_0\sqrt{\hat\varepsilon_0}\sqrt{\gconc} + c(1+\SK_0)\sqrt{\hvarepsilon}
\notag\\*&\qquad
  + ct_0\hvarepsilon
  + c\sqrt{t_0}\sqrt{\hvarepsilon}(\SK_4 + \SK_5)\sqrt{\gconc}
\,.
\label{EQconcentareaEQ2}
\end{align}
Integrating, squaring then using H\"older's inequality gives
\begin{align}
\Big(\gconc\Big)^2
 &\le
    \hsigma_0^2 + ct_0^2\hvarepsilon(1+\SK_0^2)
\notag\\&\qquad
  + c\hvarepsilon(t_0^{3}+t_0^4)(1 + \SK_0^2 + \SK_4^2 + \SK_5^2)\int_0^{t_0}\gconc\,dt\,.
\label{EQconcentareaEQ21}
\end{align}
Applying a version of Gronwall's inequality \cite[Theorem 5]{Dragomir} and supposing $\hsigma_0<1$ we find
\begin{equation}
\gconc
 \le c_4(\hsigma_0 + \hvarepsilon)(1+t_0)^4
\,,
\label{EQconcentareaEQ3}
\end{equation}
where $c_4$ is a finite, absolute constant, depending only on $c$ from \eqref{EQconcentareaEQ21} and $\SK_0$, $\SK_4$, $\SK_5$.
The maximality of $t_0$ implies that
\[
\gconc\Big|_{t=t_0} \ge \csupp\,.
\]
From \eqref{EQconcentareaEQ3} and \eqref{EQt0} we have
\[
\gconc
<
     c_4(\hsigma_0 + \hvarepsilon)\bigg(1+ \sqrt[4]{\frac{ \csupp}{c_4(\hsigma_0+\hat\varepsilon_0)}}-1\bigg)^4
= \csupp
\,,
\]
contradicting the maximality of $t_0$. Therefore
\[
t_0 \ge \sqrt[4]{\frac{{\csupp}}{c_4(\hsigma_0+\hvarepsilon)}}-1\,,
\]
as required.
To see that this proves the applicability of Theorem \ref{TMhs} on $[0,t_0)$, let $u:\Sigma\rightarrow\R$ be any
function with compact support satisfying in addition
\[
\gamma(p) = 1\quad\text{if}\quad u(p) > 0\,.
\]
Clearly then
\[
|\Sigma|_{[u>0]} \le |\Sigma|_{[\gamma=1]}
                 \le |\Sigma|_{\gamma^4} \le \csupp\,,
\]
so that both \eqref{EQambientcurvass} and \eqref{EQambientinjass} are satisfied.
\end{proof}

\begin{rmk}
The above estimate will also be applicable to the case where $f$ is an entire solution to \eqref{EQwf}. This is because we have kept all estimates local and not
used any bound on $|\Sigma|$. If $|\Sigma|$ is initialy bounded, a global estimate is easy to obtain via a simplified version of the above argument.
\end{rmk}

\begin{rmk}
For submanifolds of a Riemannian space, it makes sense to consider rescalings which scale both the ambient
space and the submanifold simultaneously by the same factor.
The concentration of area is not scale invariant: under a scaling with magnitude $\rho$, it scales like
$\rho^2$.  The Ricci curvature scales like $\rho^{-2}$, and so the right hand side of
\eqref{EQinitsmallconcareaass} also like $\rho^2$.
This implies that despite the concentration of area not being scale invariant, the condition
\eqref{EQinitsmallconcareaass} is.
\end{rmk}

By combining Propositions \ref{PRconcentestvanilla} and \ref{PRconcentarea} we conclude the following
concentration of curvature estimate for solutions in a space with some positive curvature.

\begin{cor}[Control on the concentration of curvature]
\label{CYconcentestvanillapositivecurvature}
Under the assumptions of Proposition \ref{PRconcentarea}, for any $t\in[0,T^*)$ we have
\begin{equation*}
  \begin{split}
  &\int_{[\gamma=1]} |A|^2 d\mu
 + \int_0^t\int_{[\gamma=1]} \big(|\nabla_{(2)}A|^2 + |\nabla A|^2|A|^2 + |A|^6\big)\,d\mu\,d\tau \\
&\qquad\qquad\qquad
 \le \int_{[\gamma > 0]} |A|^2 d\mu\Big|_{t=0}
 + c\hvarepsilon t,
  \end{split}
\end{equation*}
where $c$ is a constant depending only on the metric of $N$, $\cg$, $\SK_i$ for $i\in\{0,1,2,3,4,5\}$, and $\vn{\nabla_{(i)} A}^2_{2}\big|_{t=0}$ for $i\in\{0,1,2\}$.
\end{cor}

\begin{rmk}
If $T<\infty$, then we may always take $\rho$ small enough in \eqref{EQconcentsmallsupass} to force $T^* = T$. We may therefore assume, up to imposing $\rho$
small enough, that Corollary \ref{CYconcentestvanillapositivecurvature} above and Corollary \ref{CYpointwisebound2} below hold on the entire time interval
$[0,T)$, so long as $T<\infty$.
If $T=\infty$ then for any compact subinterval $I$ there is a $\rho_0 > 0$ small enough such that the pointwise estimates hold on $I$.
In the case of a sequence of compact subintervals $I_j$ where $\sup \{t\in I_j\} \rightarrow \infty$, $\rho_0 \rightarrow 0$ and the estimates do not hold in
the limit.
\end{rmk}

The Hoffman-Spruck inequality is the only critical ingredient needed for Propositions \ref{PRbootstrappingfirststep} and \ref{PRbootstrappingsecondstep}. We
therefore obtain a pointwise curvature estimate also in this case.

\begin{cor}
\label{CYpointwisebound2}
Under the assumptions of Proposition \ref{PRconcentarea}, for any $t\in[0,T^*)$ we have
\begin{equation*}
\vn{A}_{\infty,[\gamma=1]} \le c\hvarepsilon\,,
\end{equation*}
where $c$ is a constant depending only on the metric of $N$, $\cg$, $\SK_i$ for $i\in\{0,1,2,3,4,5\}$, and $\vn{\nabla_{(i)} A}^2_{2}\big|_{t=0}$ for $i\in\{0,1,2\}$.
\end{cor}


\section{Proof of the main theorems}

Let us assume first the hypotheses of Theorem \ref{TMlifespanvanilla}.  We shall describe the modifications
required to prove Theorem \ref{TMlifespanvanillapositivecurvature} after we finish the proof of Theorem
\ref{TMlifespanvanilla}.

We make the definition
\begin{equation}
\label{e:epsfuncdef}
\eta(t) = \sup_{x\in N}\int_{f^{-1}(B_\rho(x))} |A|^2d\mu.
\end{equation}
By covering $B_\rho$ with several parallel transported copies of $B_{\sfrac{\rho}{2}}$ there is a constant
$c_{\eta}$ depending only on the metric of $N$ such that
\begin{equation}
\label{e:epscovered}
\eta(t) \le c_{\eta}\sup_{x\in N}\int_{f^{-1}(B_{\sfrac{\rho}{2}}(x))} |A|^2d\mu.
\end{equation}
By Theorem \ref{TMstevanilla} we have that $f(M\times[0,t])$ is compact for $t<T$ and so the function
$\eta:[0,T)\rightarrow\R$ is continuous.  We now define
\begin{equation}
\label{e:struweparameter}
t_0
=
  \sup\{0\le t\le\min(T,\lambda) : \eta(\tau)\le 3c_{\eta}\veThree
                                         \ \text{ for }\ 0\le\tau\le t\}\,,
\end{equation}
where $\lambda$ is a parameter to be specified later. We assume that $\veThree \le \min\{\veZero,\veOne,\veTwo\}$.

The proof continues in three steps. First, we show that it must be the case that $t_0 = \min(T,\lambda)$ for a finite parameter $\lambda$ to be chosen.
Second, we show that if $t_0 = \lambda$, then we can conclude the lifespan theorem. Finally, we prove by contradiction that if $T \ne \infty$, then $t_0 \ne
T$. Note that of course if $T$ were infinite, then $\lambda < T$ and by step 2 we conclude the lifespan theorem. We label these steps as
\begin{align}
\label{e:7}
t_0 &= \min(T,\lambda),\\
\label{e:8}
t_0 &= \lambda \quad\Longrightarrow\quad\text{lifespan theorem},\\
\label{e:9}
T &\ne \infty\hskip+2.2mm \Longrightarrow\quad t_0 \ne T.
\end{align}
We now give the proof of \eqref{e:7}.
From the hypothesis of Theorem \ref{TMlifespanvanilla},
\[
\eta(0)\le \veThree \le \veTwo
 <
3c_{\eta}\veTwo\,,
\]
and therefore by the definition \eqref{e:struweparameter} and smoothness of the flow we have $t_0 > 0$.
Assume for the sake of contradiction that $t_0 < \min(T,\lambda)$.
Then \eqref{e:struweparameter} and the continuity of $\eta$ gives
\begin{equation}
\label{e:t0ltmin}
\eta(t_0) =
3c_{\eta}\veThree\,.
\end{equation}
Now set $\gamma$ to be a cutoff function as in \eqref{EQgamma} such that
\[
\chi_{B_{\sfrac{\rho}{2}}(x)} \le \tilde{\gamma} \le \chi_{B_\rho(x)},
\]
for some $x\in N$.
By varying $\rho$ in \eqref{EQsmallconcentrationcondition} we may choose $\veThree$ small enough that
$3c_\eta\veThree = \veTwo$.
Definition \eqref{e:struweparameter} then implies that the smallness condition \eqref{EQconcentass} is
satisfied on $[0,t_0)$.
Therefore we may apply Proposition \ref{PRconcentestvanilla} to obtain
\begin{equation}
\label{e:useprop44}
  \begin{split}
  \int_{f^{-1}(B_{\sfrac{\rho}{2}}(x))} |A|^2 d\mu
         &\le \int_{f^{-1}(B_\rho(x))} |A|^2d\mu\bigg|_{t=0} + c_1\veTwo t
      \\ &\le \veThree + 3c_1c_{\eta}\veThree t_0
      \\ &\le \veThree + 3c_1c_\eta\veThree \lambda
      \\ &\le 2\veThree,\qquad \text{for}\quad\lambda = \frac{1}{3c_1c_\eta}\,,
  \end{split}
\end{equation}
for all $t \in [0,t_0)$.
We combine this with \eqref{e:epscovered} to conclude
\begin{equation}
\label{e:prop44pluscovering}
\eta(t)
 \le \sup_{x\in N} \int_{f^{-1}(B_{\sfrac{\rho}{2}}(x))} |A|^2 d\mu
 \le
2c_{\eta}\veThree\,,
\end{equation}
where $0\le t < t_0$.
Since $\eta$ is continuous, we can let $t\rightarrow t_0$.
This is in contradiction with \eqref{e:t0ltmin}.  Therefore, with the choice of $\lambda$ in equation
\eqref{e:useprop44}, the assumption that $t_0 < \min(T,\lambda)$ is incorrect.
We have thus proven \eqref{e:7}.

This argument also establishes \eqref{e:8}.
Making explicit the dependence of $c_1$ on $\rho$ we have
\[
\lambda = \frac{1}{3\tilde{c}_1c_\eta}\rho^4
\]
where $\tilde{c}_1$ is now a constant depending only on $\SK_i$ for $i=0,1,2$.
Now if $t_0 = \lambda$ then
\[
  T\ge\lambda = \frac{1}{3\tilde{c_1}c_\eta^2}\rho^4\,
\]
which is \eqref{EQmaximaltimeestimatevanilla}.
Also, \eqref{e:prop44pluscovering} implies \eqref{EQconcentrationestimatevanilla}.  That is,
we have proved if $t_0 = \lambda$, then the lifespan theorem holds, which is the second step
\eqref{e:8}.

It only remains to prove equation \eqref{e:9}.
We assume
\[
t_0=T\ne\infty\,,
\]
and will derive a contradiction with the finite maximality of $T$.
The recent existence result due to Lamm and Koch \cite{kochlamm} applied for the Willmore flow of $C^0$ graphs in $\R^n$ with small Lipschitz constant, and
gives global existence. It is a straightforward adaptation of their argument to obtain local existence for $C^{0,1}$ graphs with any Lipschitz constant.
By considering graphs over the initial datum instead of graphs over $\R^2$, this additionally holds for the setting we consider here: closed immersed
submanifolds of a smooth Riemannian space.
The maximality of $T$ in light of this local existence theorem implies that as $t\rightarrow T$ either the Lipschitz constant of $f$ is becoming unbounded, or
the surface is no longer a surface.

We now claim that any sequence of immersions $f(\cdot,t_j)$ where $t_j \rightarrow T$ has a convergent subsequence in the $C^1$ topology, with limit a $C^1$
immersed surface. This is an immediate consequence of the pointwise bound for the second fundamental form from Corollary \ref{CYpointwisebound1} combined with
an application of the main compactness theorem in \cite{CooperCompactness}.
This implies that the Lipschitz constant of $f$ is not blowing up and that the geometry of the surface is not devolving; i.e., the evolving metrics are
uniformly bounded.
We thus have the desired contradiction to the finite maximality of $T$, establishing \eqref{e:9} and proving the theorem.

\qed

The modifications required to prove Theorem \ref{TMlifespanvanillapositivecurvature} are the following.  As the ambient space is no longer simply-connected and
non-positively curved, the assumptions \eqref{EQambientcurvass}, \eqref{EQambientinjass} are no longer automatically satisfied along the flow.  We have to use
Proposition \ref{PRconcentarea} to first control the concentration of the area, and then use this to allow us to apply Theorem \ref{TMhs} and conclude our
estimates.

Observe that the constant $\lambda = \frac{1}{3\tilde{c}_1c_\eta}$ is universal, depending only on the metric of $N$ and $\SK_i$ for $i=0,1,2$.
Additionally observe that the argument for step 3 above need only be performed in the case where $T < \lambda$.
We may without loss of generality assume that $\rho < \rho_0$ (an allowable choice for $\rho_0$ would be the extrinsic diameter of $\Sigma$) and then force
$\hsigma_0$ and $\hvarepsilon$ small enough (by taking $\rho$ smaller in the hypotheses of Theorem \ref{TMlifespanvanillapositivecurvature}) so that Proposition
\ref{PRconcentarea} implies the conditions \eqref{EQambientcurvass}, \eqref{EQambientinjass} are satisfied on $[0,\lambda\rho_0^4)$, which is a big enough
interval to perform the entire argument above.
In place of Proposition \ref{PRconcentestvanilla} we use Corollary \ref{CYconcentestvanillapositivecurvature}, and Corollary \ref{CYpointwisebound2} holds on
the interval $[0,\lambda\rho_0^4) \supset [0,\lambda\rho)$, which is all that is needed.
The argument goes through exactly as above.
\qed



\section*{Acknowledgements}

During the completion of this work the second author was supported by an Alexander-von-Humboldt fellowship at the
Otto-von-Guericke Universit\"at Magdeburg.
The first and third author were supported by grant ME 3816/1-1 by the Deutsche Forschungsgemeinschaft at Potsdam
Universit\"at.
Part of this work was additionally completed while the second and third authors were supported by the ARC Discovery
Project grant DP120100097 at the University of Wollongong.

\bibliographystyle{plain}
\bibliography{mbib}

\end{document}